\documentclass[12pt,reqno]{amsart}
\usepackage[a4paper,bindingoffset=0.1in,left=1in,right=1in,top=1in,bottom=1in,footskip=.25in]{geometry}
\usepackage{indentfirst,amssymb,amsmath,amsthm}    
\usepackage{newtxtext,newtxmath}
\usepackage{setspace}
\usepackage{times}
\usepackage[utf8]{inputenc}
\usepackage[T1]{fontenc}
\usepackage{verbatim}
\usepackage{hyperref}
\hypersetup{colorlinks=true, linkcolor=blue,citecolor=blue, urlcolor=blue}
\DeclareMathOperator{\dime}{dim}

\usepackage{mathtools}

\numberwithin{equation}{section}
\newtheorem{example}{Example}[section]
\newtheorem{theo}{Theorem}[section]
\numberwithin{theo}{subsection}

\newtheorem{lemma}[theo]{Lemma}

\begin{document}

\setcounter{page}{1} 
\baselineskip .65cm 
\pagenumbering{arabic}

\title{Classification of right nilpotent $\mathbb{F}_p$-braces of cardinality $p^5$.}

\author[Snehashis Mukherjee]{
Snehashis Mukherjee}

\address {\newline Snehashis Mukherjee
\newline School of Mathematical Sciences, \newline Ramakrishna Mission Vivekananda Educational and Research Institute (rkmveri), \newline Belur Math, Howrah, Box: 711202, West Bengal, India.}
\email{\href{mailto:tutunsnehashis@gmail.com}{tutunsnehashis@gmail.com}}
\begin{abstract} 
In this article the right nilpotent $\mathbb{F}_p$-braces of cardinality $p^5$ has been classified. We use the connection between nilpotent $\mathbb{F}_p$-braces of cardinality $p^5$ and nilpotent pre-Lie algebras of the same order, building on the known relationship between pre-Lie algebras and braces. Leveraging insights from the classification of nilpotent pre-Lie algebras over $\mathbb{F}_p$, we aim to provide a comprehensive classification of right nilpotent $\mathbb{F}_p$-braces of cardinality $p^5$.
\end{abstract}

\maketitle
\section{Introduction}	
Let $X$ be an arbitrary set and $R:X\times X\rightarrow X\times X$ a bijective map. The pair $(X,R)$ is said to be a set theoretic solution of the quantum Yang-Baxter equation if 
\[R_{12}R_{23}R_{12}=R_{23}R_{12}R_{23}\] holds in the set of all the maps from $X\times X\times X$ to itself, $R_{ij}$ is just $R$ acting on the $i$th and $j$th components of $X\times X\times X$ and identity on the remaining one.
\par The study of set-theoretical solutions of quantum Yang-Baxter equation was initiated by Drinfeld \cite{di} and persued by several authors \cite{e1,e2,g1,g2}.
\par  A brace is a triple $(A, +, \circ)$ where $(A, +)$ is an abelian group,
$(A, \circ)$ is a group and
\[a \circ (b + c) + a = a \circ b + a \circ c\]
for all $a, b, c \in A$. We refer to $(A, \circ)$ as the multiplicative group of the brace.
\par In 2007, W. Rump introduced braces as a generalization of Jacobson radical rings to facilitate the examination of solutions to the Yang-Baxter equation \cite{r1}. Subsequently, the study of braces has garnered considerable attention, revealing connections to various mathematical concepts. Noteworthy associations include integral group rings \cite{sy}, Garside groups \cite{ch}, groups with bijective 1-cocycles \cite{c2,do}, quantum groups \cite{do,di}, and trusses \cite{br}, among others.
\par Skew braces, introduced by L. Guarnieri and L. Vendramin in 2016 as an extension of braces, were designed to explore noninvolutive solutions to the Yang-Baxter equation \cite{gu}. Further investigations into skew braces were conducted in \cite{b1,so}. The characterization of skew braces with a cardinality of $p^3$ is provided in \cite{k1}, while a comprehensive description of all non-right nilpotent $\mathbb{F}_p$-braces with a cardinality of $p^4$ is presented in \cite{k2}. Finally, Puljic accomplished the classification of all braces with a cardinality of $p^4$ \cite{pp}.

\par The significance of right nilpotency in braces arises from its correlation with the finite multipermutation level of the associated set-theoretic solutions. The comprehension and exploration of this class of solutions, particularly those involving right nilpotent braces, have been thoroughly studied and well-documented in the literature \cite{b2,c1,c3,g3,g4}.
\par A pre-Lie algebra is a triple $(A, +, \cdot)$ consisting of a $k$-vector
space $A$ with a binary operation $(x, y) \rightarrow x \cdot y$ such that
\[(a \cdot b) \cdot c - a \cdot (b \cdot c) = (b \cdot a) \cdot c - b \cdot (a \cdot c),\]
\[(ia + jb) \cdot c = i(a \cdot c) + j(b \cdot c)\] and \[a \cdot (ib + jc) = i(a \cdot b) + j(a \cdot c)\]
for all $a, b, c \in A$ and $i, j \in k$. We say that a pre-Lie algebra is nilpotent (or strongly nilpotent) if for some
$n \in \mathbb{N}$ any product of $n$ elements is zero. We denote the radical chains in pre-Lie
algebras the same way as in braces, but using the binary operation of the pre-Lie
algebra.
\par
The association between pre-Lie algebras over $\mathbb{R}$ and left nilpotent $\mathbb{R}$-braces was elucidated in \cite{r2}, accompanied by a methodology for constructing a brace from a pre-Lie algebra. This connection underwent further scrutiny in \cite{so2}, where a formula detailing the correspondence between strongly nilpotent $\mathbb{F}_p$-braces and nilpotent pre-Lie algebras over $\mathbb{F}_p$ was presented. Additionally, in \cite{so3}, it was demonstrated that nilpotent pre-Lie rings of cardinality $p^n$ correspond to strongly nilpotent braces of the same cardinality, under the condition $p > n + 1$. These braces share the same additive group as the corresponding pre-Lie algebra and can be explicitly derived from the pre-Lie algebra by constructing the group of flows. Every brace $A$, not necessarily right nilpotent, gives rise to a related pre-Lie ring, associated with the factor brace $A/ann(p^2)$ \cite{sh}, and this construction is reversible \cite{so4}. The question of whether every brace of cardinality $p^n$ corresponds to a pre-Lie ring remains unclear.
\par Consequently, the classification of nilpotent pre-Lie rings of cardinality $p^n$ leads to the classification of strongly nilpotent braces of cardinality $p^n$, for a sufficiently large prime $p$. Skew braces and all non-right nilpotent $\mathbb{F}_p$-braces of cardinality $p^4$ were detailed in \cite{k2}. Furthermore, it was established in \cite{dora} that all braces of cardinality $p^4$, except those constructed in \cite{k2}, are right nilpotent. This implies that the remaining braces of cardinality $p^4$ awaiting classification are right nilpotent, for which there exists a connection to nilpotent pre-Lie rings. The classification of nilpotent pre-Lie algebras of cardinality $p^4$ over the field $\mathbb{F}_p$ generated by a single element has been accomplished in \cite{so2}, leading to the classification of strongly nilpotent $\mathbb{F}_p$-braces generated by a single element of the same cardinality.
\par We say that a left brace $A$ is an $\mathbb{F}_p $-brace if
its additive group is an $\mathbb{F}_p $-vector space such that
\[a*(\alpha b)=\alpha(a*b)\] for all $a,b \in A$ and $\alpha \in \mathbb{F}_p$. Note that an $\mathbb{F}_p$-brace of cardinality $p^5$
is a brace with additive group $C_p^5$. In this paper, we classify all $\mathbb{F}_p$-braces by addressing the nilpotent pre-Lie algebras of cardinality $p^5$. In this paper all braces considered are left braces.
\par \textbf{Arrangement:} The paper is structured as follows: Section $2$ provides preliminaries discussing the established connection between braces and pre-Lie algebras by Smoktunowicz \cite{so3,so4}. Consequently, our focus narrows to the classification of pre-Lie algebras of cardinality $p^5$. In Section $3$, we explore the nilpotency index of pre-Lie algebras of cardinality $p^5$, demonstrating that for any nilpotent pre-Lie algebra $(A,+,\cdot)$ over $\mathbb{F}_p$ with cardinality $p^5$, its nilpotency index is at most $8$. Section $4$ is dedicated to the classification of all potential nilpotent pre-Lie algebras generated by a single element. Sections $5$, $6$, and $7$ extend this classification to nilpotent pre-Lie algebras generated by $2$, $3$, and $4$ elements, respectively.
\section{Preliminaries}
A brace $A$ is left nilpotent if there exists $n \in \mathbb{N}$ such that $A^n = 0$, where
$A^{i+1} = A * A^i$ and $A^1 = A$. A brace is right nilpotent if there exists $n \in \mathbb{N}$
such that $A^{(n)} = 0$, where $A^{(i+1)} = A^{(i)} * A$ and $A^{(1)} = A$. A brace is strongly
nilpotent if there exists $n \in \mathbb{N}$ such that $A^{[n]} = 0$, where $A^{[i+1]}= \sum_{j=1}^{i} A^{[j]} *
A^{[i+1-j]}$ and $A^{[1]} = A$. The smallest such $n$ is called the nilpotency index of $A$.
We say that a pre-Lie ring is nilpotent (or strongly nilpotent) if for some
$n \in \mathbb{N}$ any product of $n$ elements is zero. We denote the radical chains in pre-Lie
rings the same way as in braces, but using the binary operation of the pre-Lie
ring.
\par Braces of cardinality $p^5$ are left nilpotent \cite{r1}. Then the following theorem is useful
\begin{theo}\cite[Theorem 3.1]{so5}
    Let $(A,+,\circ)$ be a left brace. If $m, n $ are natural numbers
and $A^n = A^{(m)} = 0$, then $A^{[s]} = 0$ for some natural number $s$. 
\end{theo}This theorem implies that left nilpotent braces that
are right nilpotent are strongly nilpotent. In the case where $A$ is a right nilpotent brace with a cardinality of $p^5$, the following theorem can be applied to systematically construct a pre-Lie algebra
\begin{theo}\cite[Theorem 6]{so3}
    Let $p>2$ be a prime number. Let $A$ be a strongly nilpotent brace with nilpotency index $k$ and cardinality $p^n$
for some prime number $p$, and some natural numbers
$k, n$ such that $k, n + 1 < p$. Define the binary operation $\cdot$ on $A$ as follows
\[a\cdot b=\sum_{i=0}^{p-2}\zeta^{p-1-i}\left((\zeta^ia)*b\right)\] for $a,b\in A$, where $\zeta^ia$ denotes the sum of $\zeta^i$ copies of element $a$. Then
\[(a\cdot b)\cdot c-a\cdot(b\cdot c)=(b \cdot a)\cdot c-b\cdot(a\cdot c)\] for every $a,b,c \in A$.
\end{theo}
Utilizing the Rump correspondence \cite{r1}, we have the capability to define an $\mathbb{F}_p$-brace $(A, +, \circ)$ that shares the same addition as the pre-Lie algebra $A$ and employs multiplication $\circ$ as defined in the group of flows. It is further assumed that $A$ is a nilpotent pre-Lie algebra, and the characteristic of $\mathbb{F}_p$ exceeds the nilpotency index of $A$. Using notations from \cite{ma} we have 
\begin{itemize}
    \item Let $a \in A$, and let $L_a : A \rightarrow A$ denote the left multiplication by $a$, so $L_a(b) = a \cdot b$.
Define $ L_c \cdot L_b(a) = L_c(L_b(a)) = c \cdot (b \cdot a)$. Define
\[e^{L_a}(b)=b+a\cdot b+\frac{1}{2!}a\cdot(a\cdot b)+\frac{1}{3!}a\cdot\left(a\cdot(a\cdot b)\right)+\cdots\]
\item We can formally consider element $1$ such that $1 \cdot a = a\cdot 1 = a$ in our pre-Lie algebra
(as in \cite{ma}) and define
\[W(a)=e^{L_a}(1)-1=a+\frac{1}{2!}a\cdot a+\frac{1}{3!}a\cdot\left(a\cdot a\right)+\cdots\]Notice that $W(a) : A \rightarrow A$ is a bijective function, provided that $A$ is a nilpotent
pre-Lie algebra.
\item  Let $\Omega(a) : A \rightarrow A$ be the inverse function to the function $W(a)$, so $\Omega(W(a)) =
W(\Omega(a)) = a$. Following \cite{ma} the first terms of $\Omega$ are
\[\Omega(a)=a+\frac{1}{2}a\cdot a+\frac{1}{4}(a\cdot a)\cdot a+\cdots\]
\item Define 
\[a\circ b=a+e^{L_{\Omega(a)}}(b)\] Here the addition is same as in pre-Lie algebra. 
\end{itemize}
Therefore, we establish a one-to-one correspondence between nilpotent pre-Lie rings of cardinality $p^n$ and braces of identical cardinality. In the forthcoming sections, our emphasis is on the construction of a nilpotent pre-Lie algebra with a cardinality of $p^5$, where $(A, +)$ is isomorphic to $C_p^5$, consequently achieving the creation of right nilpotent $\mathbb{F}_p$-braces with a cardinality of $p^5$.

\section{Nilpotency Index} Before delving into the nilpotency index of pre-Lie algebra we will show that there exists a pre-Lie algebra $A$ of cardinality $p^5$ such that $A^{[6]} 
 \subset A^{[5]} = A^{[4]}$ are non-zero and $A^{[7]} = 0$.
 \begin{example}
     Consider the pre-Lie algebra $A$ with basis $\{\alpha_1,\ldots,\alpha_5\}$ over the field $\mathbb{F}_p$ together with the multiplications as follows
     \[\alpha_1 \cdot \alpha_1=\alpha_2,~\alpha_1 \cdot \alpha_4=\alpha_5,~\alpha_2 \cdot \alpha_1=\alpha_3,~\alpha_2 \cdot \alpha_3=\alpha_4,~\alpha_3 \cdot \alpha_3=-\alpha_5,~\alpha_4 \cdot \alpha_1=\alpha_5,\]
     \[\alpha_i \cdot \alpha_j=0~\text{for all $i,j$ otherwise}.\]
    it is easy to check with the multiplication table that all of the pre-Lie
algebra relations
\[(a \cdot b) \cdot c - a \cdot (b \cdot c) = (b \cdot a) \cdot c - b \cdot (a \cdot c)\] via taking $a,b,c$ as basis elements.
\par Note that 
\begin{align*}
    A^{[2]}&=\mathbb{F}_p \alpha_2+\mathbb{F}_p \alpha_3+\mathbb{F}_p \alpha_4+\mathbb{F}_p \alpha_5.\\
    A^{[3]}&=\mathbb{F}_p \alpha_3+\mathbb{F}_p \alpha_4+\mathbb{F}_p \alpha_5.\\
    A^{[4]}&=\mathbb{F}_p \alpha_4+\mathbb{F}_p \alpha_5.\\
     A^{[5]}&=\mathbb{F}_p \alpha_4+\mathbb{F}_p \alpha_5.\\
      A^{[6]}&=\mathbb{F}_p \alpha_5.\\
       A^{[7]}&=0.
\end{align*}

 \end{example}
\subsection{Nilpotency Index} We start this subsection by stating the following theorem.
\begin{theo}\label{nil}
    Let $(A,+,\cdot)$ be a nilpotent pre-Lie algebra over $\mathbb{F}_p$ of cardinality $p^5$. Then $A^{[8]}=0$.
\end{theo}
\begin{proof}
    Consider the sets $A^{[5]}, A^{[4]}, A^{[3]},A^{[2]}, A^{[1]}=A$. Here the following two cases arise\\
    \textbf{Case I:} Assume 
    \[A^{[5]} \neq A^{[4]} \neq A^{[3]} \neq A^{[2]} \neq A^{[1]}=A\]
    If any of these sets equal $0$, we are done. So let us assume all of them are non zero. Then for each $i=1,2,3,4$, $A^{[i]}/A^{[i+1]}$ has cardinality $p$ and $A^{[5]}$
has also cardinality $p$. We will show that in this case $A^{[8]}=0$.\\
On contrary assume $A^{[8]}\neq 0$. Then we must have $A^{[8]}=A^{[7]}=A^{[6]}=A^{[5]}$.
Now,
$A^{[8]}=\sum_{i=1}^{7}A^{[i]}\cdot A^{[8-i]}$.\\
Note $A \cdot A^{[7]}=A \cdot A^{[8]}\subset A^{[9]},~ A^{[2]} \cdot A^{[6]}=A^{[2]} \cdot A^{[7]}\subset A^{[9]},~A^{[3]} \cdot A^{[5]}=A^{[3]} \cdot A^{[6]}\subset A^{[9]}$. Similarly, $A^{[5]} \cdot A^{[3]}, ~A^{[6]} \cdot A^{[2]},~A^{[7]} \cdot A \subset A^{[9]}$. Now let us consider $A^{[4]} \cdot A^{[4]}$. For $x \in A^{[4]}, y \in A^{[3]}, z \in A$, we have from 
\begin{equation}\label{im}
    (x \cdot y) \cdot z - x \cdot (y \cdot z) = (y \cdot x) \cdot z - y \cdot (x \cdot z)
\end{equation}
 $x \cdot (y \cdot z) \in A^{[7]} \cdot A + A^{[3]} \cdot A^{[5]} \subset A^{[9]}$.\\
Similarly if $x \in A^{[4]}, y \in A^{[2]}, z \in A^{[2]}$, $x \cdot (y \cdot z) \in A^{[6]} \cdot A^{[2]} + A^{[2]} \cdot A^{[6]} \subset A^{[9]}$ and if  $x \in A^{[4]}, y \in A, z \in A^{[3]}$, $x \cdot (y \cdot z) \in A^{[5]} \cdot A^{[3]} + A \cdot A^{[7]} \subset A^{[9]}$. So $A^{[4]} \cdot A^{[4]} \subset A^{[9]}$. Hence $A^{[8]} \subset A^{[9]}$. Using decomposition of $A^{[9]}$, we can show $A^{[9]} \subset A^{[10]}$. Continuing we get \[A^{[5]} \subset A^{[6]} \subset A^{[7]} \subset A^{[8]} \subset A^{[9]} \subset A^{[10]}\subset \cdots \]
As $A$ is nilpotent we have $A^{[5]}=0$, a contradiction. Hence $A^{[8]}=0$.\\
\textbf{Case II:} Let there exists $i \in \{1,2,3,4\}$ such that $A^{[i]}=A^{[i+1]}$. Note if $A=A^{[2]}$, then for every $j$, $A^{[j]} \subset A^{[j+1]}$ and we have $A=0$.\\
If $A^{[2]}=A^{[3]}$, then for every $j \geq 2$, $A^{[j]}=A^{[j+1]}$ and hence $A^{[2]}=0$.\\
If $A^{[3]}=A^{[4]}$, then we claim $A^{[3]}=0$. Note $A^{[4]}=A\cdot A^{[3]}+A^{[2]}\cdot A^{[2]}+A^{[3]}\cdot A$. Since $A^{[3]}=A^{[4]}$, we have $A\cdot A^{[3]}, A^{[3]}\cdot A \subset A^{[5]}$. Now from $x \in A^{[2]},y,z \in A$, we get form 
\[(x \cdot y) \cdot z - x \cdot (y \cdot z) = (y \cdot x) \cdot z - y \cdot (x \cdot z)\] that $x\cdot (y \cdot z) \in A\cdot A^{[3]}+ A^{[3]}\cdot A \subset A^{[5]}$. Hence $A^{[2]}\cdot A^{[2]}\subset A^{[5]}$. So $A^{[4]} \subset A^{[5]}$. Using induction we can show that $A^{[j]}=A^{[j+1]}~ \forall j \geq 3$ implying $A^{[3]}=0$.\\
If $A^{[4]}=A^{[5]}$, we claim $A^{[8]}=0$. Note if $A^{[4]}=A^{[5]}=A^{[6]}$, then $A \cdot A^{[5]}, A^{[2]} \cdot A^{[4]}, A^{[4]} \cdot A^{[2]}, A^{[5]} \cdot A \subset A^{[7]}$. If $x \in A^{[3]},y \in A^{[2]},z \in A$, from \ref{im} we have $x \cdot (y \cdot z) \in A^{[5]}\cdot A^{[2]}+A^{[2]}\cdot A^{[4]} \subset A^{[7]}$. Similarly for $x \in A^{[3]},y \in A,z \in A^{[2]}$, from \ref{im} we have $x \cdot (y \cdot z) \in A^{[4]}\cdot A^{[2]}+A\cdot A^{[5]} \subset A^{[7]}$. So $A^{[3]} \cdot A^{[3]} \subset A^{[7]} $.\\
Note that $A^{[6]}=\sum_{i=1}^{5}A^{[i]}\cdot A^{[6-i]}$ implies $A^{[6]} \subset A^{[7]}$. Using induction we have $A^{[j]}=A^{[j+1]}$ for any $j \geq 4$ and hence $A^{[4]}=0$. Now if $A^{[6]} \neq A^{[5]}=A^{[4]} $, then $|A^{[6]}|=p$. If $A^{[8]} \neq 0$, then $A^{[8]}=A^{[7]}=A^{[6]}\neq A^{[5]}=A^{[4]}$.\\
We claim that $A^{[3]} \cdot A^{[2]} \subset A^{[6]}$. Indeed for $x \in A^{[3]}, y, z \in A$ from \ref{im} $x \cdot (y \cdot z) \in A^{[4]} \cdot A +A \cdot A^{[4]}=A^{[5]} \cdot A +A \cdot A^{[5]} \subset A^{[6]}$.\\
Note $A \cdot A^{[7]}, A^{[2]}\cdot A^{[6]}, A^{[4]} \cdot A^{[4]}, A^{[6]} \cdot A^{[2]}, A^{[7]} \cdot A \subset A^{[9]}$. Now for $x \in A^{[5]}, y \in A^{[2]}, z \in A$, from \ref{im} we have $x\cdot (y \cdot z) \in A^{[7]} \cdot A + A^{[2]} \cdot A^{[6]} \subset A^{[9]}$. Similarly if $x \in A^{[5]}, y \in A, z \in A^{[2]}$, from \ref{im} we have $x\cdot (y \cdot z) \in A^{[6]} \cdot A^{[2]} + A \cdot A^{[7]} \subset A^{[9]}$. So $A^{[5]} \cdot A^{[3]} \subset A^{[9]}$.\\
Now for $x \in A^{[3]}, y \in A^{[4]}, z \in A$, again from \ref{im}, we have $x\cdot (y \cdot z) \in A^{[7]} \cdot A + A^{[4]} \cdot A^{[4]}= A^{[7]} \cdot A + A^{[4]} \cdot A^{[5]}\subset A^{[9]}$.\\
Similarly for $x \in A^{[3]}, y \in A^{[3]}, z \in A^{[2]}$ we have  $x\cdot (y \cdot z) \in A^{[3]} \cdot ( A^{[3]} \cdot A^{[2]}) \subset A^{[3]} \cdot  A^{[6]}\subset A^{[9]}$.\\
Similarly for $x \in A^{[3]}, y \in A^{[2]}, z \in A^{[3]}$, again from \ref{im}, we have $x\cdot (y \cdot z) \in A^{[5]} \cdot A^{[3]} + A^{[2]} \cdot A^{[6]} \subset A^{[9]}$.\\
Finally for $x \in A^{[3]}, y \in A, z \in A^{[4]}$, again from \ref{im}, we have $x\cdot (y \cdot z) \in   A^{[4]} \cdot A^{[4]}+A \cdot A^{[7]}=  A^{[4]} \cdot A^{[5]}+A \cdot A^{[7]}  \subset A^{[9]}$.\\
Finally $A^{[5]} \cdot A^{[3]} \subset A^{[9]}$. Note $A^{[8]}=\sum_{i=1}^{7}A^{[i]}\cdot A^{[8-i]}$ implies $A^{[8]} \subset A^{[9]}$. Using induction we can prove that $A^{[j]}=A^{[j+1]}$ for all $j \geq 6$ which in turn imply $A^{[6]}=0$, a contradiction. Hence $A^{[8]}=0$.
\end{proof}
\section{pre-Lie algebra $A$ with $(A,+)\cong C_p^5$ and generated by one element as a pre-Lie algebra}
We start this section by mentioning the following result.
\begin{theo}
    Let  $(A,+)\cong C_p^5$ and generated by a single element $x$ as a pre-Lie algebra. Then $A^{[4]} \neq 0$.
\end{theo}
\begin{proof}
    On contrary let us assume $A^{[4]}=0$. Then $A$ is generated by $x,x^2,x\cdot x^2,x^2 \cdot x$ as a $\mathbb{F}_p$-vector space. Hence $\dime_{\mathbb{F}_p}A \leq 4$.
\end{proof}
\subsection{$A^{[7]} \neq 0 $ and $A^{[8]}=0$}
\begin{lemma}\label{l1}
    In pre-Lie algebra $A$ with $A^{[5]} \neq A^{[4]}$, we have
    \[A\cdot A^{[6]}=A^{[6]} \cdot A=A^{[2]} \cdot A^{[5]}=A^{[2]}\cdot A^{[5]}=A^{[4]} \cdot A^{[3]}=0\]Moreover the only possible non-zero elements in $A^{[7]}$ are the following
    \[(x^2 \cdot x)\cdot\left((x^2 \cdot x)\cdot x\right),(x \cdot x^2)\cdot\left((x^2 \cdot x)\cdot x\right),(x \cdot x^2)\cdot\left((x \cdot x^2)\cdot x\right)\]
\end{lemma}
\begin{proof}
    As $A^{[5]} \neq A^{[4]}$, we have $A^{[7]} = A^{[6]}=A^{[5]}\neq A^{[4]}\neq A^{[3]}\neq A^{[2]}\neq A$.\\
    Note $A \cdot A^{[6]}=A \cdot A^{[7]}\subset A^{[8]}=0$.Similarly $A^{[6]} \cdot A=A^{[2]} \cdot A^{[5]}=A^{[2]}\cdot A^{[5]}=0$.\\
    Now for $A^{[4]}\cdot A^{[3]}$, we have the following
    \begin{itemize}
        \item For $x \in A^{[4]}, y \in A^{[2]},z \in A$, from \ref{im}, we have $x \cdot (y \cdot z) \in A^{[6]}\cdot A+A^{[2]}\cdot A^{[5]}=0.$
        \item For $x \in A^{[4]}, y \in A,z \in A^{[2]}$, from \ref{im}, we have $x \cdot (y \cdot z) \in A^{[5]}\cdot A^{[2]}+A\cdot A^{[6]}=0.$
    \end{itemize}
    Hence $A^{[4]}\cdot A^{[3]}=0$.\\
    For $A^{[3]}\cdot A^{[4]}$, we have the following
    \begin{itemize}
        \item For $x \in A^{[3]}, y \in A,z \in A^{[3]}$, from \ref{im}, we have $x \cdot (y \cdot z) \in A^{[4]}\cdot A^{[3]}+A\cdot A^{[6]}=0.$
        \item For $x \in A^{[4]}, y,z \in A^{[2]}$, from \ref{im}, we have $x \cdot (y \cdot z) \in A^{[5]}\cdot A^{[2]}+A^{[2]}\cdot A^{[5]}=0.$
    \end{itemize}
    So the only possible non zero terms in $A^{[3]}\cdot A^{[4]}$ are of the form $x \cdot (y \cdot z)$ where $x,y \in A^{[3]},z \in A$. Now from pre-Lie algebra relation we have,
    \[\left((x \cdot x^2)\cdot(x^2 \cdot x)\right).x-(x \cdot x^2)\cdot\left((x^2 \cdot x).x\right)=\left((x^2 \cdot x)\cdot(x \cdot x^2)\right).x-(x^2 \cdot x)\cdot\left((x \cdot x^2).x\right)\]
    Note $\left((x \cdot x^2)\cdot(x^2 \cdot x)\right).x,\left((x^2 \cdot x)\cdot(x \cdot x^2)\right).x \in A^{[6]}\cdot A=0$. So we have
    \[(x \cdot x^2)\cdot\left((x^2 \cdot x).x\right)=(x^2 \cdot x)\cdot\left((x \cdot x^2).x\right).\]
    So the only non possible non zero elements in $A^{[7]}$ are $(x^2 \cdot x)\cdot\left((x^2 \cdot x)\cdot x\right),(x \cdot x^2)\cdot\left((x^2 \cdot x)\cdot x\right),(x \cdot x^2)\cdot\left((x \cdot x^2)\cdot x\right)$.\\
    In $A^{[6]}$, $A\cdot A^{[5]}=A\cdot A^{[7]}\subset A^{[8]}=0$. Similarly $A^{[5]}\cdot A=0$.\\
    For $A^{[2]}\cdot A^{[4]}$, we have
    \begin{itemize}
        \item For $x\in A^{[2]}, y \in A^{[3]}, z \in A$, $x\cdot(y\cdot z)\in A^{[5]}\cdot A+A^{[3]}\cdot A^{[3]}=A^{[3]}\cdot A^{[3]}$.
        \item For $x\in A^{[2]}, y \in A, z \in A^{[3]}$, $x\cdot(y\cdot z)\in A\cdot A^{[5]}+A^{[3]}\cdot A^{[3]}=A^{[3]}\cdot A^{[3]}$.
    \end{itemize}
    Note that for $x^2\cdot(x^2\cdot x^2)$ from \[(x^2\cdot x)\cdot x-x^2\cdot x^2=(x \cdot x^2)\cdot x -x \cdot (x^2\cdot x)\] we have 
    \begin{align*}\label{eq1}
        x^2\cdot\left((x^2\cdot x)\cdot x\right)-x^2\cdot(x^2\cdot x^2)&=x^2\cdot\left((x \cdot x^2)\cdot x\right) -x^2\cdot\left(x \cdot (x^2\cdot x)\right)\\
        x^2\cdot(x^2\cdot x^2)&=x^2\cdot\left((x^2\cdot x)\cdot x\right)+x^2\cdot\left(x \cdot (x^2\cdot x)\right)-x^2\cdot\left((x \cdot x^2)\cdot x\right)\\
        x^2\cdot(x^2\cdot x^2)&=(x^2\cdot x)\cdot (x^2\cdot x)+(x^2\cdot x)\cdot (x^2\cdot x)-(x \cdot x^2)\cdot(x^2\cdot x).
    \end{align*}
 Now for $x \in A^{[4]},y,z\in A$, $x\cdot(y\cdot z)\in A^{[5]}\cdot A+A\cdot A^{[5]}=0 $. So $A^{[4]}\cdot A^{[2]}=0$.\\  
 Now for $A^{[3]}\cdot A^{[3]}$, we have for $x\in A^{[3]}, y \in A,z\in A^{[2]}$, $x\cdot(y\cdot z)\in A^{[4]}\cdot A^{[2]}+A\cdot A^{[5]}=0$. So $(x^2\cdot x)\cdot (x \cdot x^2)=(x \cdot x^2)\cdot (x \cdot x^2)=0$.
\end{proof}
\begin{lemma}\label{l2}
    In pre-Lie algebra $A$ with $A^{[5]}= A^{[4]}$, we have
    \[A\cdot A^{[6]}=A^{[6]} \cdot A=A^{[3]} \cdot A^{[4]}=A^{[4]} \cdot A^{[3]}=A^{[5]}\cdot A^{[2]}=0\]Moreover the only possible non-zero elements in $A^{[7]}$ are the following
    \[x^2\cdot(x^2 \cdot (x\cdot x^2)),x^2\cdot(x^2 \cdot (x^2\cdot x))\]
\end{lemma}
\begin{proof}
    Here $A^{[7]} = A^{[6]}\neq A^{[5]}= A^{[4]}\neq A^{[3]}\neq A^{[2]}\neq A$.\\
    Note $A \cdot A^{[6]}=A \cdot A^{[7]}\subset A^{[8]}=0$. Similarly $A^{[6]} \cdot A=A^{[3]} \cdot A^{[4]}=A^{[4]}\cdot A^{[3]}=0$.\\
    For $A^{[5]}\cdot A^{[2]}$, note if $x \in A^{[5]},y,z \in A$, then $x \cdot(y \cdot z)\in A^{[6]}\cdot A+A \cdot A^{[6]}=0$.Hence $A^{[5]}\cdot A^{[2]}=0$.\\
    For $A^{[2]}\cdot A^{[5]}$ we have
    \begin{itemize}
        \item If $x \in A^{[2]}, y \in A, z \in A^{[4]}$, then $x \cdot(y \cdot z) \in A^{[3]}\cdot A^{[4]}+A \cdot A^{[6]}=0$.
        \item If $x \in A^{[2]}, y \in A^{[4]}, z \in A$, then $x \cdot(y \cdot z) \in A^{[6]}\cdot A+A^{[4]} \cdot A^{[3]}=0$.
        \item If $x \in A^{[2]}, y \in A^{[3]}, z \in A^{[2]}$, then $x \cdot(y \cdot z) \in A^{[5]}\cdot A^{[2]}+A^{[3]} \cdot A^{[4]}=0$.
    \end{itemize}
    Hence the only possible non zero elements in $A^{[7]}$ are $x^2\cdot(x^2 \cdot (x\cdot x^2)),x^2\cdot(x^2 \cdot (x^2\cdot x))$.
    \par Also note due to the fact $A^{[4]}=A^{[5]}$, we have 
    $A^{[4]}\cdot A=A^{[5]}\cdot A\subset A^{[6]}$. Similarly $A\cdot A^{[4]}\subset A^{[6]}$. For $x\in A^{[3]}, y,z \in A$, we have $x\cdot (y \cdot z)\in A^{[4]}\cdot A+A\cdot A^{[4]}\subset A^{[6]}$. Also for $x\in A^{[2]}, y \in A,z \in A^{[2]}, x\cdot(y\cdot z)\in A^{[3]}\cdot A^{[2]}+A\cdot A^{[4]}\subset A^{[6]}$. So the only element of $A^{[5]}$ which is not in $A^{[6]}$ is $x^2\cdot (x^2 \cdot x)$.
\end{proof}
\begin{theo}
    Let $(A,+,\cdot)$ be a pre-Lie algebra with $(A,+)\cong C_{p}^5$ and $A$ is generated by a single element $x$. Moreover let $A^{[7]} \neq 0$ and $A^{[5]} \neq A^{[4]}$. Then we can say the following
    \begin{itemize}
        \item[1.] Elements $x,x^2,x^2\cdot x,(x^2\cdot x)\cdot x, (x^2\cdot x) \cdot \left((x^2\cdot x)\cdot x\right)$ form a basis of the pre-Lie algebra $A$ as a $\mathbb{F}_p$ vector space.
        \item[2.] For some $\alpha \in \mathbb{F}_p^{\times}$ we have $(x^2\cdot x) -\alpha(x \cdot x^2)\in A^{[4]}$.
        \item[3.] $(x \cdot x^2)\cdot x-\alpha(x^2 \cdot x)\cdot x \in A^{[5]}$, $x\cdot(x^2\cdot x), x\cdot(x\cdot x^2) \in A^{[5]}$. 
        \item[4.] In $A^{[5]}$ the elements $x\cdot \left(x\cdot(x\cdot x^2)\right),x\cdot \left(x\cdot(x^2\cdot x)\right),\left(x\cdot(x\cdot x^2)\right)\cdot x,\left(x\cdot(x^2\cdot x)\right)\cdot x$ are $0$. Also the following relations hold in $A^{[5]}$
        \begin{itemize}
            \item[(i)] $x\cdot\left((x \cdot x^2)\cdot x\right)=\alpha ~x\cdot\left((x^2 \cdot x)\cdot x\right), \left((x \cdot x^2)\cdot x\right)\cdot x=\alpha\left((x^2 \cdot x)\cdot x\right)\cdot x$.
            \item[(ii)]  $x\cdot(x^2\cdot x^2)=(1-\alpha)~x\cdot\left((x^2\cdot x)\cdot x\right)$, $(x^2\cdot x^2)\cdot x=(1-\alpha)\left((x^2\cdot x)\cdot x\right)\cdot x$.
            \item[(iii)] $(x^2 \cdot x)\cdot x^2=\left((x^2 \cdot x)\cdot x\right)\cdot x+x\cdot\left((x^2\cdot x)\cdot x\right),(x \cdot x^2)\cdot x^2=\alpha(x^2 \cdot x)\cdot x^2$
            
        \end{itemize}
         Also $x^2\cdot (x\cdot x^2)-\alpha x^2\cdot(x^2\cdot x) \in A^{[6]}$.
\item[5.] $(x\cdot x^2)\cdot (x^2 \cdot x)=\alpha (x^2 \cdot x)\cdot (x^2 \cdot x), x^2\cdot(x^2\cdot x^2)=(2-\alpha)(x^2\cdot x)\cdot (x^2\cdot x)$. 
\item[6.] $A^{[7]}=A^{[6]}=A^{[5]}=\mathbb{F}_p~\left((x^2\cdot x) \cdot \left((x^2\cdot x)\cdot x\right)\right), A^{[4]}/A^{[5]}=\mathbb{F}_p~\left((x^2\cdot x)\cdot x\right),\\ A^{[3]}/A^{[4]}= \mathbb{F}_p~\left(x^2\cdot x\right)$. 
    \end{itemize}
\end{theo}
\begin{proof}
    By Lemma \ref{l1}, we know the only possible non zero elements in $A^{[7]}$ are $(x^2 \cdot x)\cdot\left((x^2 \cdot x)\cdot x\right),(x \cdot x^2)\cdot\left((x^2 \cdot x)\cdot x\right),(x \cdot x^2)\cdot\left((x \cdot x^2)\cdot x\right)$.
Reasoning as in Lemma \ref{l1} we obtain $A^{[3]}/A^{[4]}$ has dimension $1$ as a vector space over the field $\mathbb{F}_{p}$. We claim $x^2\cdot x \notin A^{[4]}$ and hence a non zero element in $A^{[3]}/A^{[4]}$. Indeed if $x^2\cdot x \in A^{[4]}$, then $(x^2 \cdot x)\cdot\left((x^2 \cdot x)\cdot x\right)=(x \cdot x^2)\cdot\left((x^2 \cdot x)\cdot x\right)=0$. Also from 
\begin{equation}\label{e1}
    (x^2\cdot x)\cdot x-x^2\cdot x^2=(x \cdot x^2)\cdot x -x \cdot (x^2\cdot x)
\end{equation}
we have 
\[(x \cdot x^2)\cdot\left((x^2\cdot x)\cdot x\right)-(x \cdot x^2)\cdot(x^2\cdot x^2)=(x \cdot x^2)\cdot\left((x \cdot x^2)\cdot x\right) -(x \cdot x^2)\cdot\left(x \cdot (x^2\cdot x)\right)\]
Note $(x \cdot x^2)\cdot(x^2\cdot x^2) \in A^{[3]}\cdot (A^{[2]}\cdot A^{[2]}) =0$ and $(x \cdot x^2)\cdot\left(x \cdot (x^2\cdot x)\right) \in A^{[3]}\cdot (A\cdot A^{[3]})=0$. So we end up getting $(x \cdot x^2)\cdot\left((x \cdot x^2)\cdot x\right)=0$ which implies $A^{[7]}=0$, a contradiction. Hence there exists $\alpha \in \mathbb{F}_p$ such that 
\begin{equation}\label{e2}
(x \cdot x^2)-\alpha(x^2 \cdot x) \in A^{[4]}.
\end{equation}

So we have $(x^2 \cdot x)\cdot\left((x \cdot x^2)\cdot x\right)-\alpha(x^2 \cdot x)\cdot\left((x^2 \cdot x)\cdot x\right) \in A^{[8]}=0$ implying $(x^2 \cdot x)\cdot\left((x \cdot x^2)\cdot x\right)=\alpha(x^2 \cdot x)\cdot\left((x^2 \cdot x)\cdot x\right)$. \\
Similarly\[(x \cdot x^2)\cdot\left((x \cdot x^2)\cdot x\right)=\alpha(x \cdot x^2)\cdot\left((x^2 \cdot x)\cdot x\right)=\alpha(x^2 \cdot x)\cdot\left((x \cdot x^2)\cdot x\right)=\alpha^2(x^2 \cdot x)\cdot\left((x^2 \cdot x)\cdot x\right).\]
So clearly $A^{[7]}=\mathbb{F}_p((x^2 \cdot x)\cdot\left((x^2 \cdot x)\cdot x\right))$.\\
Note that $(x^2\cdot x)\cdot x \in A^{[4]}\setminus A^{[5]}$ otherwise $(x^2 \cdot x)\cdot\left((x^2 \cdot x)\cdot x\right)=0$ implying $A^{[7]}=0$. From Lemma \ref{l1} we have $A^{[4]}/A^{[5]}$ is one dimensional as $\mathbb{F}_p$ vector space and hence we can assume $A^{[4]}/A^{[5]}=\mathbb{F}_p\left((x^2\cdot x)\cdot x\right)$. So we must have \[x\cdot(x^2\cdot x)-\beta(x^2\cdot x)\cdot x \in A^{[5]}\] implying
\[(x^2\cdot x)\cdot\left(x\cdot(x^2\cdot x)\right)=\beta(x^2\cdot x)\cdot\left((x^2\cdot x)\cdot x\right). \] As $(x^2\cdot x)\cdot\left(x\cdot(x^2\cdot x)\right)=0$ and $(x^2\cdot x)\cdot\left((x^2\cdot x)\cdot x\right) \neq 0$, we have $\beta=0$. So $x\cdot(x^2\cdot x) \in A^{[5]}$. Similarly $x\cdot(x\cdot x^2) \in A^{[5]}$. Now from \ref{e2}, we have 
\[(x \cdot x^2)\cdot x-\alpha(x^2 \cdot x)\cdot x \in A^{[5]}.\]
Now in $A^{[5]}$, due to the fact that $x\cdot(x^2\cdot x),x\cdot(x\cdot x^2) \in A^{[5]}$ and $A \cdot A^{[5]}=A^{[5]}\cdot A=0$, we have \[x\cdot \left(x\cdot(x\cdot x^2)\right)=x\cdot \left(x\cdot(x^2\cdot x)\right)=\left(x\cdot(x\cdot x^2)\right)\cdot x=\left(x\cdot(x^2\cdot x)\right)\cdot x=0.\]
Due to similar reason we have
\[x\cdot\left((x \cdot x^2)\cdot x\right)=\alpha ~x\cdot\left((x^2 \cdot x)\cdot x\right), \left((x \cdot x^2)\cdot x\right)\cdot x=\alpha\left((x^2 \cdot x)\cdot x\right)\cdot x.\]
now from \ref{e1},
\[x\cdot\left((x^2\cdot x)\cdot x\right)-x\cdot(x^2\cdot x^2)=x\cdot\left((x \cdot x^2)\cdot x\right) -x\cdot(x \cdot (x^2\cdot x))\] implying
$x\cdot(x^2\cdot x^2)=(1-\alpha)x\cdot\left((x^2\cdot x)\cdot x\right)$. Similarly $(x^2\cdot x^2)\cdot x=(1-\alpha)\left((x^2\cdot x)\cdot x\right)\cdot x$. Now from
\[\left((x^2 \cdot x)\cdot x\right)\cdot x-(x^2 \cdot x)\cdot x^2=\left(x\cdot(x^2\cdot x)\right)\cdot x-x\cdot\left((x^2\cdot x)\cdot x\right)\] we have \[(x^2 \cdot x)\cdot x^2=\left((x^2 \cdot x)\cdot x\right)\cdot x+x\cdot\left((x^2\cdot x)\cdot x\right).\]
Similarly \[(x \cdot x^2)\cdot x^2=\left((x\cdot x^2)\cdot x\right)\cdot x-((x \cdot (x \cdot x^2))\cdot x=\alpha \left[\left((x^2 \cdot x)\cdot x\right)\cdot x+x\cdot\left((x^2\cdot x)\cdot x\right)\right]=\alpha(x^2 \cdot x)\cdot x^2.\]
From $x\cdot x^2-\alpha x^2\cdot x \in A^{[4]}$, we have $x^2\cdot (x\cdot x^2)-\alpha x^2\cdot(x^2\cdot x) \in A^{[6]}$.\\
Now for $A^{[6]}$ from \ref{e1}, we have \[
(x\cdot x^2)\cdot (x^2 \cdot x)-\alpha (x^2 \cdot x)\cdot (x^2 \cdot x)\in A^{[4]}\cdot A^{[3]}=0.
\] So we get \[(x\cdot x^2)\cdot (x^2 \cdot x)=\alpha (x^2 \cdot x)\cdot (x^2 \cdot x).\] From Lemma \ref{l1}, we have \[x^2\cdot(x^2\cdot x^2)=(2-\alpha)(x^2\cdot x)\cdot (x^2\cdot x).\]
\end{proof}
\begin{theo}
     Let $(A,+,\cdot)$ be a pre-Lie algebra with $(A,+)\cong C_{p}^5$ and $A$ is generated by a single element $x$. Moreover let $A^{[7]} \neq 0$ and $A^{[5]}= A^{[4]}$. Then we can say the following
     \begin{itemize}
         \item[1.] Elements $x,x^2,x^2\cdot x,x^2\cdot(x^2\cdot x), x^2\cdot\left(x^2\cdot(x^2\cdot x)\right)$ form a basis of the pre-Lie algebra $A$ as a $\mathbb{F}_p$ vector space.
         \item[2.] $x\cdot x^2 \in A^{[4]}$.
         \item[3.] There exists $\alpha_1,\alpha_2\in \mathbb{F}_p$ such that \[x\cdot(x^2\cdot x)-\alpha_1~x^2\cdot(x^2\cdot x),(x^2\cdot x)\cdot x-\alpha_2~x^2\cdot(x^2\cdot x)\in A^{[6]}. \] Also \[(x^2\cdot x)\cdot x-x^2\cdot x^2=(x \cdot x^2)\cdot x -x \cdot (x^2\cdot x).\]
         \item[4.] $x\cdot\left(x\cdot(x\cdot x^2)\right)=\left(x\cdot(x\cdot x^2)\right)\cdot x=x \cdot \left((x\cdot x^2)\cdot x\right)=\left((x\cdot x^2)\cdot x\right)\cdot x=(x\cdot x^2)\cdot x^2=0.$
         \item[5.] Possible non zero generators of $A^{[6]}$ are $x\cdot\left(x^2\cdot(x^2\cdot x)\right),\left(x^2\cdot(x^2\cdot x)\right)\cdot x, x^2\cdot\left(x\cdot(x^2\cdot x)\right),x^2\cdot\left((x^2\cdot x)\cdot x\right),x^2\cdot(x^2\cdot x^2),(x^2\cdot x)\cdot(x^2\cdot x),\left((x^2\cdot x)\cdot x\right)\cdot x^2, \left(x\cdot(x^2\cdot x)\right)\cdot x^2$. All the other generators are $0$.
         \item[6.] $A^{[7]}=A^{[6]}=\mathbb{F}_p~x^2\cdot\left(x^2\cdot(x^2\cdot x)\right)$, $A^{[5]}/A^{[6]}=A^{[4]}/A^{[6]}=\mathbb{F}_p~x^2\cdot(x^2\cdot x)$, $A^{[3]}/A^{[4]}=\mathbb{F}_p~(x^2\cdot x)$.
     \end{itemize}
\end{theo}
\begin{proof}
    Note $x^2\cdot x \notin A^{[4]}$. Indeed otherwise $x^2 \cdot (x^2\cdot x) \in A^{[6]}$ implying $A^{[5]}=A^{[6]}$ from Lemma \ref{l2}. Hence $A^{[3]}/A^{[4]}=\mathbb{F}_p )(x^2\cdot x)$. Then for some $\alpha \in \mathbb{F}_p$ we have 
    \[x\cdot x^2-\alpha~x^2\cdot x \in A^{[4]}.\] This implies 
    \[x^2\cdot(x\cdot x^2)-\alpha~x^2\cdot(x^2\cdot x) \in A^{[6]}.\] Now from Lemma \ref{l2}, $x^2\cdot(x\cdot x^2) \in A^{[6]}$. So $\alpha=0$ implying $x\cdot x^2\in A^{[4]}$. Hence the only non zero element in $A^{[7]}$ is $x^2\cdot\left(x^2\cdot(x^2\cdot x)\right)$.\\
    In $A^{[4]}$, $x\cdot(x\cdot x^2),(x\cdot x^2)\cdot x \in A^{[6]}$. Also as $A^{[4]}/A^{[6]}=A^{[5]}/A^{[6]}=\mathbb{F}_p~x^2(x^2\cdot x)$, there exists $\alpha_1,\alpha_2\in \mathbb{F}_p$ such that \[x\cdot(x^2\cdot x)-\alpha_1~x^2\cdot(x^2\cdot x),(x^2\cdot x)\cdot x-\alpha_2~x^2\cdot(x^2\cdot x)\in A^{[6]}. \] Also \[(x^2\cdot x)\cdot x-x^2\cdot x^2=(x \cdot x^2)\cdot x -x \cdot (x^2\cdot x).\]
    Now in $A^{[5]}$, \[x\cdot\left(x\cdot(x\cdot x^2)\right)=\left(x\cdot(x\cdot x^2)\right)\cdot x=x \cdot \left((x\cdot x^2)\cdot x\right)=\left((x\cdot x^2)\cdot x\right)\cdot x=0.\] Also $(x\cdot x^2)\cdot x^2\in A^{[5]}\cdot A^{[2]}=0$. From Lemma \ref{l2}, all the other elements except $x^2\cdot(x^2\cdot x)$ are in $A^{[6]}$. \\
    In $A^{[6]}$, every element in $A\cdot A^{[5]}$ except $x\cdot\left(x^2\cdot(x^2\cdot x)\right)$ are $0$. Similarly every element in $ A^{[5]}\cdot A$ except $\left(x^2\cdot(x^2\cdot x)\right)\cdot x$ are $0$. Also as $x\cdot x^2 \in A^{[4]}$, we have 
    \begin{align*}
        x^2\cdot\left(x\cdot(x\cdot x^2)\right)=x^2\cdot\left((x\cdot x^2)\cdot x^2\right)&=(x\cdot x^2)\cdot(x\cdot x^2)=(x\cdot x^2)\cdot(x^2\cdot x)=\\
        (x^2\cdot x)\cdot (x\cdot x^2)=\left((x\cdot x^2)\cdot x\right)\cdot x^2&=\left(x\cdot ((x\cdot x^2)\right)\cdot x^2=0
    \end{align*}
    \end{proof}
    \subsection{$A^{[7]}=0$ and $A^{[6]}\neq 0$.} We start this subsection by stating the following theorem.
    \begin{theo}\label{im3}
        In pre-Lie algebra $A$ with $A^{[5]} \neq A^{[4]}$, we have
    $A\cdot A^{[5]}=A^{[5]}\cdot A=  A^{[4]} \cdot A^{[2]}=0$. Moreover $A^{[6]}$ is generated by $(x^2\cdot x)\cdot (x^2 \cdot x)$ and $(x\cdot x^2)\cdot (x^2 \cdot x)$ as a $\mathbb{F}_p$ vector space.
    \end{theo}
    \begin{proof}
        Here $A^{[6]}=A^{[5]}\neq A^{[4]}\neq A^{[3]}\neq A^{[2]} \neq A$. Now $A\cdot A^{[5]}=A\cdot A^{[6]}\subset A^{[7]}=0$. Similarly $A^{[5]}\cdot A=0$. For $x\in A^{[4]},y,z\in A, x\cdot(y\cdot z)\in A^{[5]}\cdot A+A\cdot A^{[5]}=0$. Hence $A^{[4]}\cdot A^{[2]}=0$. Also if $x\in A^{[3]}, y\in A,z\in A^{[2]}, x\cdot(y\cdot z)\in A^{[4]}\cdot A^{[2]}+A\cdot A^{[5]}=0$. So $A^{[3]}\cdot(A\cdot A^{[2]})=0$. Consequently, $(x^2\cdot x)\cdot(x\cdot x^2)=(x\cdot x^2)\cdot(x\cdot x^2)=0$.\\ From \ref{im} we have 
        \begin{align*}
            &x^2\cdot\left((x^2\cdot x)\cdot x\right)=(x^2\cdot x)\cdot (x^2 \cdot x),~~~x^2\cdot\left((x\cdot x^2)\cdot x\right)=(x\cdot x^2)\cdot (x^2 \cdot x)\\
            &x^2\cdot \left(x\cdot(x \cdot x^2)\right)=(x^2\cdot x)\cdot(x\cdot x^2)-(x\cdot x^2)\cdot(x\cdot x^2)=0.\\
            &x^2\cdot \left(x\cdot(x^2 \cdot x)\right)=(x^2\cdot x)\cdot(x^2\cdot x)-(x\cdot x^2)\cdot(x^2\cdot x).\\
            &x^2\cdot(x^2\cdot x^2)=x^2\cdot\left((x^2\cdot x)\cdot x\right)-x^2\cdot\left(x\cdot(x^2\cdot x)\right)=(x^2\cdot x)\cdot (x^2 \cdot x)-(x\cdot x^2)\cdot (x^2 \cdot x).
        \end{align*}
        Hence $A^{[6]}$ is generated by $(x^2\cdot x)\cdot (x^2 \cdot x)$ and $(x\cdot x^2)\cdot (x^2 \cdot x)$.
    \end{proof}
    \begin{theo}
     Let $(A,+,\cdot)$ be a pre-Lie algebra with $(A,+)\cong C_{p}^5$ and $A$ is generated by a single element $x$. Moreover let $A^{[6]} \neq 0, A^{[7]} =0$ and $A^{[5]}\neq A^{[4]}$. Then we can say the following
     \begin{itemize}
         \item[1.] Elements $x,x^2,x^2\cdot x,(x^2\cdot x)\cdot x, x^2\cdot\left((x^2\cdot x)\cdot x\right)$ form a basis of the pre-Lie algebra $A$ as a $\mathbb{F}_p$ vector space.
         \item[2.] $x\cdot x^2 \in A^{[4]}$.
         \item[3.] $x\cdot(x\cdot x^2),(x\cdot x^2)\cdot x, x\cdot(x^2\cdot x)-(x^2\cdot x)\cdot x\in A^{[5]}$.
         \item[4.]  In $A^{[5]}$ \begin{align*}
         x\cdot\left((x^2\cdot x)\cdot x\right)=x\cdot\left( x\cdot (x^2\cdot x)\right)&,~~\left((x^2\cdot x)\cdot x\right)\cdot x=\left(x \cdot(x^2\cdot x)\right)\cdot x. \\
         x\cdot\left((x\cdot x^2)\cdot x\right)=x\cdot\left( x\cdot (x\cdot x^2)\right)&=~\left((x\cdot x^2)\cdot x\right)\cdot x=\left(x \cdot(x\cdot x^2)\right)\cdot x=0.\\
         (x^2\cdot x^2)\cdot x=2~\left((x^2\cdot x)\cdot x\right)\cdot x&, x\cdot (x^2\cdot x^2)=2~ x\cdot\left((x^2\cdot x)\cdot x\right).\\
         x^2\cdot(x\cdot x^2)&=(x^2\cdot x)\cdot x^2+x\cdot(x^2\cdot x^2).
        \end{align*}
        \item[5.] $A^{[6]}=A^{[5]}=\mathbb{F}_p~x^2\cdot\left((x^2\cdot x)\cdot x\right)$, $A^{[4]}/A^{[5]}=\mathbb{F}_p~(x^2\cdot x)\cdot x$, $A^{[3]}/A^{[4]}=\mathbb{F}_p~(x^2\cdot x)$.
     \end{itemize}
     \end{theo}
     \begin{proof}
         From Theorem \ref{im3}, we have $\dime_{\mathbb{F}_p}A^{[3]}/A^{[4]}=1$. Note $x^2\cdot x\notin A^{[4]}$. Otherwise $(x^2\cdot x)\cdot (x^2 \cdot x)=(x\cdot x^2)\cdot (x^2 \cdot x)=0$ implying $A^{[6]}=0$.\\
         So $A^{[3]}/A^{[4]}=\mathbb{F}_p~x^2\cdot x$. Hence there exists $\alpha\in \mathbb{F}_p$ such that \[x\cdot x^2-\alpha~x^2\cdot x \in A^{[4]}.\] This implies 
         \[(x^2\cdot x)\cdot(x\cdot x^2)=\alpha~(x^2\cdot x)\cdot(x^2\cdot x) \]
         As $(x^2\cdot x)\cdot(x\cdot x^2)=0$, either $\alpha=0$ or $(x^2\cdot x)\cdot(x^2\cdot x)=0$. If $\alpha \neq 0$, then \[0=(x\cdot x^2)\cdot(x\cdot x^2)=\alpha~(x\cdot x^2)\cdot(x^2\cdot x)\] implies $(x\cdot x^2)\cdot(x^2\cdot x)=0$ contradicting the fact $A^{[6]}\neq 0$. Hence $\alpha=0$. So $x\cdot x^2 \in A^{[4]}$. From Theorem \ref{im3}, we have 
         \[x^2\cdot\left((x^2\cdot x)\cdot x\right)=x^2\cdot \left(x\cdot(x^2 \cdot x)\right)=x^2\cdot(x^2\cdot x^2)=(x^2\cdot x)\cdot (x^2 \cdot x).\]
         Note from Theorem \ref{im3}, we have $\dime_{\mathbb{F}_p}A^{[4]}/A^{[5]}=1$. $x\cdot(x^2\cdot x)\notin A^{[5]}$. Otherwise $(x^2\cdot x)\cdot(x^2\cdot x)=x^2\cdot \left(x\cdot(x^2 \cdot x)\right)\in A^{[7]}=0$. \\
        For some $\alpha\in \mathbb{F}_p$ we must have \[(x^2\cdot x)\cdot x-\alpha~x\cdot(x^2\cdot x)\in A^{[5]}. \] This implies 
        \[x^2\cdot\left((x^2\cdot x)\cdot x\right)=\alpha~x^2\cdot\left(x\cdot(x^2\cdot x)\right)\] implying $\alpha=1$.\\
        Now as $A^{[5]}\cdot A=A\cdot A^{[5]}=0$, we have 
        \begin{align*}
         x\cdot\left((x^2\cdot x)\cdot x\right)=x\cdot\left( x\cdot (x^2\cdot x)\right)&,~~\left((x^2\cdot x)\cdot x\right)\cdot x=\left(x \cdot(x^2\cdot x)\right)\cdot x. \\
         x\cdot\left((x\cdot x^2)\cdot x\right)=x\cdot\left( x\cdot (x\cdot x^2)\right)&=~\left((x\cdot x^2)\cdot x\right)\cdot x=\left(x \cdot(x\cdot x^2)\right)\cdot x=0.
        \end{align*}
        From \[(x^2\cdot x)\cdot x-x^2\cdot x^2=(x \cdot x^2)\cdot x -x \cdot (x^2\cdot x),\] we have 
        \[(x^2\cdot x^2)\cdot x=2~\left((x^2\cdot x)\cdot x\right)\cdot x, x\cdot (x^2\cdot x^2)=2~ x\cdot\left((x^2\cdot x)\cdot x\right).\]
        From $A^{[4]}\cdot A^{[2]}=0$, $(x\cdot x^2)\cdot x^2=0$. Also from \[(x^2\cdot x)\cdot x^2-x^2\cdot(x\cdot x^2)=(x \cdot x^2)\cdot x^2-x\cdot(x^2\cdot x^2)\] we have \[x^2\cdot(x\cdot x^2)=(x^2\cdot x)\cdot x^2+x\cdot(x^2\cdot x^2).\]
     \end{proof}
      \begin{theo}\label{im3}
        In pre-Lie algebra $A$ with $A^{[5]} = A^{[4]}$, we have
    $A^{[2]}\cdot A^{[4]}= A^{[4]} \cdot A^{[2]}=0$. Moreover $A^{[6]}$ is generated by elements $x \cdot \left(x^2\cdot (x^2 \cdot x)\right)$ and $\left(x^2\cdot (x^2 \cdot x)\right)\cdot x$. 
    \end{theo}
    \begin{proof}
        In $A^{[6]}$, $A^{[4]}\cdot A^{[2]}=A^{[5]}\cdot A^{[2]}\subset A^{[7]}=0$. Similarly $A^{[2]}\cdot A^{[4]}=0$.\\
        As $A^{[4]}=A^{[5]}$, $A\cdot(A^{[4]}\cdot A), (A^{[4]}\cdot A)\cdot A\subset A^{[7]}=0$. Also for $x\in A^{[3]},y ,z \in A, x\cdot(y\cdot z)\in A^{[4]}\cdot A+A\cdot A^{[4]}\subset A^{[6]}$. Hence $A\cdot(A^{[3]}\cdot A^{[2]})\subset A^{[7]}=0$.
        Similarly $(A^{[3]}\cdot A^{[2]})\cdot A=0$. Now for $x\in A^{[2]},y\in A,z\in A^{[2]}$, from \ref{im}, $x\cdot(y\cdot z)\in A^{[3]}\cdot A^{[2]}+A\cdot A^{[4]}\subset A^{[6]}$. Hence $x^2\cdot(x\cdot x^2)\in A^{[6]}$ implying $x\cdot\left(x^2\cdot(x\cdot x^2)\right)=\left(x^2\cdot(x\cdot x^2)\right)\cdot x=0$. \\
        Now from \ref{im}, we have
        \[\left((x^2\cdot x)\cdot x\right)\cdot x^2-(x^2\cdot x)\cdot(x\cdot x^2)=\left(x\cdot(x^2\cdot x)\right)\cdot x^2-x\cdot\left((x^2\cdot x)\cdot x^2\right).\] This implies $(x^2\cdot x)\cdot(x\cdot x^2)=0$ as $\left((x^2\cdot x)\cdot x\right)\cdot x^2,\left(x\cdot(x^2\cdot x)\right)\cdot x^2 \in A^{[4]}\cdot A^{[2]}=0$ and $x\cdot\left((x^2\cdot x)\cdot x^2\right)\in A\cdot(A^{[3]}\cdot A^{[2]})=0.$ Similarly $(x\cdot x^2)\cdot(x \cdot x^2)=(x\cdot x^2)\cdot(x^2\cdot x)=0$. \\
        Again from \ref{im}, $(x^2\cdot x)\cdot(x^2\cdot x)=-\left(x^2\cdot(x^2\cdot x)\right)\cdot x$. Hence the result.
    
    \end{proof}
    \begin{theo}
       Let $(A,+,\cdot)$ be a pre-Lie algebra with $(A,+)\cong C_{p}^5$ and $A$ is generated by a single element $x$. Moreover let $A^{[6]} \neq 0, A^{[7]} =0$ and $A^{[5]}= A^{[4]}$. Then we can say the following 
       \begin{itemize}
           \item[1.] Elements $x,x^2,x^2\cdot x,x^2\cdot (x^2\cdot x), \left(x^2\cdot(x^2\cdot x)\right)\cdot x$ form a basis of the pre-Lie algebra $A$ as a $\mathbb{F}_p$ vector space.
         \item[2.] $x\cdot x^2 \in A^{[4]}$. 
         \item[3.] $x\cdot(x\cdot x^2),(x\cdot x^2)\cdot x \in A^{[6]}$.
         \item[4.]  $x\cdot\left(x\cdot(x\cdot x^2)\right)=\left(x\cdot(x\cdot x^2)\right)\cdot x=x^2\cdot (x\cdot x^2)= (x\cdot x^2)\cdot x^2=0$.\\  $x\cdot\left(x\cdot (x^2\cdot x)\right),\left(x\cdot (x^2\cdot x)\right)\cdot x,~x\cdot\left((x^2\cdot x)\cdot x\right),\left((x^2\cdot x)\cdot x\right)\cdot x, (x^2\cdot x)\cdot x^2 \in A^{[6]}$.
         \item[5.] $\left(x^2\cdot(x^2\cdot x)\right)\cdot x=x\cdot\left(x^2\cdot(x^2\cdot x)\right)=-(x^2\cdot x)\cdot(x^2\cdot x)$.
         \item[6.] $A^{[6]}=\mathbb{F}_p~(x^2\cdot x)\cdot(x^2\cdot x)$, $A^{[5]}/A^{[6]}=\mathbb{F}_p~x^2\cdot(x^2\cdot x), A^{[4]}/A^{[3]}=\mathbb{F}_p~x^2\cdot x$.
       \end{itemize}
    \end{theo}
    \begin{proof}
        Notice $\dime_{\mathbb{F}_p}A^{[3]}/A^{[4]}=1$. If $x^2\cdot x \in A^{[4]}$, then $A^{[6]}=0$, a contradiction. So there exists $\alpha \in \mathbb{F}_p$ such that $x\cdot x^2-\alpha~x^2\cdot x \in A^{[4]}$ holds. This implies $(x^2\cdot x)\cdot(x\cdot x^2)=\alpha~(x^2\cdot x)\cdot (x^2\cdot x)$. As $(x^2\cdot x)\cdot(x\cdot x^2)=0$, we have $\alpha=0$ implying $x\cdot x^2\in A^{[4]}$.\\
        As $A^{[4]}=A^{[5]}$, $x\cdot(x\cdot x^2),(x\cdot x^2)\cdot x \in A^{[6]}$.\\
        Also $x\cdot\left(x\cdot(x\cdot x^2)\right)\in A\cdot A^{[5]}=0$. Similarly $\left(x\cdot(x\cdot x^2)\right)\cdot x=0$. Notice $x\cdot (x^2\cdot x),(x^2\cdot x)\cdot x \in A^{[4]}=A^{[5]}$ implying $x\cdot\left(x\cdot (x^2\cdot x)\right),\left(x\cdot (x^2\cdot x)\right)\cdot x,~x\cdot\left((x^2\cdot x)\cdot x\right),\left((x^2\cdot x)\cdot x\right)\cdot x \in A^{[6]}$. As $x\cdot x^2\in A^{[4]}$, from Theorem \ref{im3}, $x^2\cdot (x\cdot x^2)= (x\cdot x^2)\cdot x^2=0$. Again it follows from Theorem \ref{im3}, $(x^2\cdot x)\cdot x^2 \in A^{[6]}$. So $x^2\cdot(x^2\cdot x)\in A^{[5]}\setminus A^{[6]}$.\\
       From \ref{im} we have 
       \[\left(x^2\cdot(x^2\cdot x)\right)\cdot x-x^2\cdot\left((x^2\cdot x)\cdot x\right)=\left((x^2\cdot x)\cdot x^2\right)\cdot x-(x^2\cdot x)\cdot(x^2\cdot x)\] implying $\left(x^2\cdot(x^2\cdot x)\right)\cdot x=-(x^2\cdot x)\cdot(x^2\cdot x)$. Hence from Theorem \ref{im3} \[\left(x^2\cdot(x^2\cdot x)\right)\cdot x=x\cdot\left(x^2\cdot(x^2\cdot x)\right)=-(x^2\cdot x)\cdot(x^2\cdot x).\] So $A^{[6]}=\mathbb{F}_p~(x^2\cdot x)\cdot(x^2\cdot x)$.
    \end{proof}
    \subsection{$A^{[6]}=0$ and $A^{[5]}\neq 0$.}
    \begin{theo}
       Let $(A,+,\cdot)$ be a pre-Lie algebra with $(A,+)\cong C_{p}^5$ and $A$ is generated by a single element $x$. Moreover let $A^{[5]} \neq 0, A^{[6]} =0$ and $A^{[5]} \neq  A^{[4]}$. Then we can say the following 
       \begin{itemize}
           \item[1.] $\{x,x^2,u,v,w\}$ forms a basis of $A$ as $\mathbb{F}_p$ vector space where $u \in \{x\cdot x^2,x^2\cdot x\},v=u\cdot x~\text{or}~x\cdot u$ and $w=l \cdot m$ or where $l \in \{u\cdot x, x\cdot u\}$ for some $u\in A^{[3]}$ and $m=x$ or vice-versa.
           \item[2.] For some $\alpha,\beta \in \mathbb{F}_p$, not both zero we have $\alpha~x^2\cdot x+\beta~x\cdot x^2\in A^{[4]}$.
           \item[3.] $\alpha~x\cdot(x^2\cdot x)+\beta~x\cdot(x\cdot x^2), \alpha~(x^2\cdot x)\cdot x+\beta~(x\cdot x^2)\cdot x\in A^{[5]}$.
           \item[4.] In $A^{[5]}$ we have 
           \begin{align*}
              \alpha\left(x\cdot(x^2\cdot x)\right)\cdot x+\beta\left(x\cdot(x\cdot x^2)\right)\cdot x=~&\alpha~x\cdot\left(x\cdot(x^2\cdot x)\right)+\beta~x\cdot\left(x\cdot(x\cdot x^2)\right)\cdot x=0,\\
              \alpha\left((x^2\cdot x)\cdot x\right)\cdot x+\beta\left((x\cdot x^2)\cdot x\right)\cdot x=~&\alpha~x\cdot\left((x^2\cdot x)\cdot x\right)+\beta~x\cdot\left((x\cdot x^2)\cdot x\right)=0,\\
              \alpha~(x^2\cdot x)\cdot x^2+\beta~(x\cdot x^2)\cdot x^2=&~\alpha~x^2\cdot(x^2\cdot x)+\beta~x^2\cdot(x\cdot x^2)=0.
           \end{align*}
           
       \end{itemize}
       \end{theo}
       \begin{proof}
           Here we have $A^{[5]}\neq A^{[4]}\neq A^{[3]}\neq A^{[2]}\neq A$. Since $A^{[3]}/A^{[4]}$ is one dimensional vector space over $\mathbb{F}_p$ for some $\alpha,\beta \in \mathbb{F}_p$, not both zero we have $\alpha~x^2\cdot x+\beta~x\cdot x^2\in A^{[4]}$. Consequently in $A^{[4]}$ we have 
          \[\alpha~x\cdot(x^2\cdot x)+\beta~x\cdot(x\cdot x^2), \alpha~(x^2\cdot x)\cdot x+\beta~(x\cdot x^2)\cdot x\in A^{[5]}.\]
           Similarly in $A^{[5]}$ we have 
           \begin{align*}
              \alpha~\left(x\cdot(x^2\cdot x)\right)\cdot x+\beta~\left(x\cdot(x\cdot x^2)\right)\cdot x=~&\alpha~x\cdot\left(x\cdot(x^2\cdot x)\right)+\beta~x\cdot\left(x\cdot(x\cdot x^2)\right)\cdot x=0,\\
              \alpha~\left((x^2\cdot x)\cdot x\right)\cdot x+\beta~\left((x\cdot x^2)\cdot x\right)\cdot x=~&\alpha~x\cdot\left((x^2\cdot x)\cdot x\right)+\beta~x\cdot\left((x\cdot x^2)\cdot x\right)=0,\\
              \alpha~(x^2\cdot x)\cdot x^2+\beta~(x\cdot x^2)\cdot x^2=&~\alpha~x^2\cdot(x^2\cdot x)+\beta~x^2\cdot(x\cdot x^2)=0.
           \end{align*}
           Take any non zero product of some copies of products of $x$ from $A^{[5]}$, say $w$. Then $w=l \cdot m$ or where $l \in \{u\cdot x, x\cdot u\}$ for some $u\in A^{[3]}$ and $m=x$ or vice-versa. Note that $u\cdot x \in A^{[4]}\setminus A^{[5]}$ and $u\in A^{[3]}\setminus A^{[4]}$. So $\{x,x^2,u,v,w\}$ forms a basis of $A$ as $\mathbb{F}_p$ vector space.
           
       \end{proof}
       \begin{theo}
           Let $(A,+,\cdot)$ be a pre-Lie algebra with $(A,+)\cong C_{p}^5$ and $A$ is generated by a single element $x$. Moreover let $A^{[5]} \neq 0, A^{[6]} =0$ and $A^{[5]}=  A^{[4]}$. Then we can say the following 
           \begin{itemize}
               \item[1.] $\{x,x^2,x^2\cdot x,x\cdot x^2,x^2\cdot(x^2\cdot x)\}$  forms a basis of $A$ as $\mathbb{F}_p$ vector space.
               \item[2.] $A^{[3]}/A^{[4]}=\mathbb{F}_p~x\cdot x^2+\mathbb{F}_p~x^2\cdot x$, $A^{[4]}=A^{[5]}=\mathbb{F}_p~x^2\cdot(x^2\cdot x)$.
               \item[3.] The only non-zero product of $5$ or more elements in $A$ is $x^2\cdot(x^2\cdot x)$.
           \end{itemize}
       \end{theo}
       \begin{proof}
           Using \ref{im}, we deduce $A\cdot A^{[4]}=A^{[4]}\cdot A=A^{[3]}\cdot A^{[2]}=A^{[2]}\cdot(A\cdot A^{[2]})=0$. So the only non zero product of $x$ in $A^{[5]}$ is $x^2\cdot(x^2\cdot x)$. So $A^{[4]}=A^{[5]}=\mathbb{F}_p~x^2\cdot(x^2\cdot x)$. It is clear that $A^{[3]}/A^{[4]}$ is two dimensional $\mathbb{F}_p$ vector space and hence $A^{[3]}/A^{[4]}=\mathbb{F}_p~x\cdot x^2+\mathbb{F}_p~x^2\cdot x$. So clearly $\{x,x^2,x^2\cdot x,x\cdot x^2,x^2\cdot(x^2\cdot x)\}$  forms a basis of $A$ as $\mathbb{F}_p$ vector space.
       \end{proof}
       \subsection{$A^{[5]}=0$ and $A^{[4]}\neq 0$.}
       \begin{theo}
           Let $(A,+,\cdot)$ be a pre-Lie algebra with $(A,+)\cong C_{p}^5$ and $A$ is generated by a single element $x$. Moreover let $A^{[4]} \neq 0, A^{[5]} =0$ and $\dime_{\mathbb{F}_p} A^{[4]}=1$. Then we can say the following
           \begin{itemize}
               \item[1.] $\{x,x^2,x^2\cdot x,x\cdot x^2,a\}$  forms a basis of $A$ as $\mathbb{F}_p$ vector space where $a=u\cdot x~\text{or}~x\cdot u$ and $u\in\{x^2\cdot x,x\cdot x^2\}$.
               \item[2.] $\dime_{\mathbb{F}_p}A^{[3]}/A^{[4]}=2$. 
           \end{itemize}
       \end{theo}
       \begin{proof}
           Here we have $A^{[4]}\neq A^{[3]}\neq A^{[2]}\neq A$. As $\dime_{\mathbb{F}_p} A^{[4]}=1$, we must have $\dime_{\mathbb{F}_p}A^{[3]}/A^{[4]}=2$. So $x^2\cdot x$ and $x\cdot x^2$ must form a basis of $A^{[3]}/A^{[4]}$. Now take any non zero product of copies of $x$ from $A^{[4]}$, say $a$. Note $a$ can be of the form $u\cdot x$ or $x\cdot u$ where $u=x\cdot x^2~\text{or}~x^2\cdot x$. Hence $\{x,x^2,x^2\cdot x,x\cdot x^2,a\}$  forms a basis of $A$.
       \end{proof}
          \begin{theo}
           Let $(A,+,\cdot)$ be a pre-Lie algebra with $(A,+)\cong C_{p}^5$ and $A$ is generated by a single element $x$. Moreover let $A^{[4]} \neq 0, A^{[5]} =0$ and $\dime_{\mathbb{F}_p} A^{[4]}=2$. Then we can say the following
           \begin{itemize}
               \item[1.] $\{x,x^2,u,x\cdot u,u\cdot x\}$  forms a basis of $A$ as $\mathbb{F}_p$ vector space where $u\in\{x^2\cdot x,x\cdot x^2\}$.
               \item[2.]  For some $\alpha,\beta \in \mathbb{F}_p$, not both zero we have $\alpha~x^2\cdot x+\beta~x\cdot x^2\in A^{[4]}$.
               \item[3.] $\alpha~x\cdot(x^2\cdot x)+\beta~x\cdot(x\cdot x^2)= \alpha~(x^2\cdot x)\cdot x+\beta~(x\cdot x^2)\cdot x=0$.
                \end{itemize}
       \end{theo}
       \begin{proof}
           Since $\dime_{\mathbb{F}_p}A^{[3]}/A^{[4]}=1$, for some $\alpha,\beta \in \mathbb{F}_p$, not both zero we have $\alpha~x^2\cdot x+\beta~x\cdot x^2\in A^{[4]}$. Consequently,
           \[\alpha~x\cdot(x^2\cdot x)+\beta~x\cdot(x\cdot x^2)= \alpha~(x^2\cdot x)\cdot x+\beta~(x\cdot x^2)\cdot x=0\] One of $x\cdot x^2$ and $x^2\cdot x$ must be in $A^{[3]}\setminus A^{[4]}$. Take that element as $u$. Then we claim $u\cdot x$ and $x\cdot u$ are linearly independent in $A$. If not let $u\cdot x=\gamma x\cdot u$. Let $v\in \{x\cdot x^2,x^2\cdot x\}\setminus \{u\}$.\\
           As $A^{[3]}/ A^{[4]}$ has basis $\{u\}$ for some $\nu \in \mathbb{F}_p$ we have $v-\nu~u \in A^{[4]}$. So we have $v\cdot x=\nu~u\cdot x=\nu\gamma~x\cdot u$. Similarly $x\cdot v=\nu x\cdot u$. From 
           \[(x^2\cdot x)\cdot x-x^2\cdot x^2=(x\cdot x^2)\cdot x-x\cdot(x^2\cdot x)\] it is clear that $A^{[4]}=\mathbb{F}_p~x\cdot u$, a contradiction. Hence $\{x,x^2,u,x\cdot u,u\cdot x\}$  forms a basis of $A$ as $\mathbb{F}_p$ vector space.
       \end{proof}
\section{pre-Lie algebra $A$ with $(A,+)\cong C_p^5$ and generated by two elements as a pre-Lie algebra}\label{s1}
Let us start by commenting on the nilpotency index of $(A,+,\cdot)$ which is generated by $2$ elements as a pre-Lie algebra.
\begin{theo}\label{nil2}
    Let $A$ be generated by two elements as a pre-Lie algebra. Then $A^{[2]}\neq 0$ and $A^{[6]}=0$.
\end{theo}
\begin{proof}
    If $A^{[2]}=0$ then $\dime_{\mathbb{F}_p}A \leq 2$, a contradiction. \\
    Now $A/A^{[2]}$ has dimension $2$ as a vector space over $\mathbb{F}_p$. So $\dime_{\mathbb{F}_p}A^{[2]}=3$. Note if $A^{[2]}=A^{[3]}$, then from Theorem \ref{nil} we have $A^{[2]}=0$. So $A^{[3]} \subset A^{[2]}$ and hence $\dime_{\mathbb{F}_p}A^{[3]} \leq 2$. Now the following two cases may arise:\\
    \textbf{Case I:} Let $\dime_{\mathbb{F}_p}A^{[3]} = 2$. If $A^{[3]}=A^{[4]}$, then again from Theorem \ref{nil} $A^{[3]}=0$ and hence we arrive at a contradiction.\\
    So $A^{[4]} \subset A^{[3]}$ implying $\dime_{\mathbb{F}_p}A^{[4]} \leq 1$.\\
    If $\dime_{\mathbb{F}_p}A^{[4]} < 1$, then $A^{[4]}=0$ and hence we are done.\\
    Let $\dime_{\mathbb{F}_p}A^{[4]} = 1$. If $A^{[5]} \subset A^{[4]}$ then $A^{[5]}=0$ and we are done.\\
    If $A^{[4]}=A^{[5]}$ we have the following two possibilities.
    \begin{itemize}
        \item $A^{[6]} \subset A^{[5]}=A^{[4]} \implies A^{[6]}=0$.
        \item $A^{[6]}= A^{[5]}=A^{[4]}$. Then from Theorem \ref{nil}
 we have $A^{[4]}=0$, a contradiction.
 \end{itemize}
 \textbf{Case II:} Let $\dime_{\mathbb{F}_p}A^{[3]} < 2$. If $\dime_{\mathbb{F}_p}A^{[3]} = 0$, we are done.\\
 Assume $\dime_{\mathbb{F}_p}A^{[3]} = 1$. If $A^{[3]}=A^{[4]}$ then $A^{[3]}=0$, a contradiction. If $A^{[4]} \subset A^{[3]}$ then $A^{[4]}=0$.\\
 So in any case $A^{[6]}=0$.
 \end{proof}
\subsection{$A^{[2]} \neq 0$ and $A^{[3]}=0$.} 
Assume $A$ is generated by $x$ and $y$ as a pre-Lie algebra. As $x$ and $y$ are generators $x,y \notin A^{[2]}$. Also no linear combination of $x$ and $y$ can belong to $A^{[2]}$. So $\dime_{\mathbb{F}_p}A/A^{[2]}\geq 2$. We claim that $\dime_{\mathbb{F}_p}A/A^{[2]}= 2$. Let $\overline{u} \in A/A^{[2]}$ with $u \neq x,y$. As $u \in A$ and $A$ is generated by $x$ and $y$ as a pre-Lie algebra we have
\[u=a_1x+a_2y+a_3~x\cdot x+a_4~x\cdot y+a_5~ y \cdot x+a_6~ y \cdot y.\] where $a_i\in \mathbb{F}_p$ for all $i=1,\ldots,6$. This implies 
\[\overline{u}=a_1\overline{x}+a_2\overline{y}\]. So any member $A/A^{[2]}$ is a linear combination of $\overline{x}$ and $\overline{y}$. So $\dime_{\mathbb{F}_p}A/A^{[2]}= 2$.\\
Now $A^{[2]}$ is generated by $x \cdot x,x\cdot y, y \cdot x, y \cdot y$ as a $\mathbb{F}_p$ vector space. Clearly $\dime_{\mathbb{F}_p}A^{[2]}= 3$. So choose $u,v,w \in \{x \cdot x,x\cdot y, y \cdot x, y \cdot y\}$ such that they are linearly independent. Then $\{x,y,u,v,w\}$ forms a basis of the $\mathbb{F}_p$ vector space A. For any $i,j\in\{x,y\}$ 
\[i \cdot j=\alpha_{ij}u+\beta_{ij}v+\gamma_{ij}w\] where at least $3$ tuples $(\alpha_{ij},\beta_{ij},\gamma_{ij})$ are not $(0,0,0)$ or multiples of each other.\\
So we can summarize this subsection by stating the following theorem.
\begin{theo}
    Let $A$ be a pre-Lie algebra generated by two elements $x$ and $y$ as a pre-Lie algebra. Assume $A^{[2]} \neq 0$ and $A^{[3]}=0$. Then we have the 
following
\begin{itemize}
    \item[1.] $A$ has a $\mathbb{F}_p$ basis $\{x,y,u,v,w\}$ where $u,v,w \in \{x \cdot x,x\cdot y, y \cdot x, y \cdot y\}$.
    \item[2.]   For any $i,j\in\{x,y\}$ 
\[i \cdot j=\alpha_{ij}u+\beta_{ij}v+\gamma_{ij}w\] where at least $3$ tuples $(\alpha_{ij},\beta_{ij},\gamma_{ij})$ are not $(0,0,0)$ or multiples of each other.
\end{itemize}
\end{theo}
\subsection{$A^{[3]} \neq 0 $ and $A^{[4]}=0$.}
By similar argument as in Subsection \ref{s1} we have $\dime_{\mathbb{F}_p}A/A^{[2]}= 2$. So $\dime_{\mathbb{F}_p}A^{[2]}= 3$. So $\dime_{\mathbb{F}_p}A^{[3]}\leq 2$. Here we have the following two cases.\\
\textbf{Case I:} $\dime_{\mathbb{F}_p}A^{[3]}= 2$. Define $S:=\{x \cdot x,x\cdot y, y \cdot x, y \cdot y\}$. Note that all of $S$ can not be in $A^{[3]}$ as this will imply $A^{[2]} \subset A^{[3]}$ and hence $A^{[2]}=0$, a contradiction. So at least one element of $S$ is not in $A^{[3]}$. Let us call it $u$.\\
As $A^{[3]} \neq 0$ and $\dime_{\mathbb{F}_p}A^{[3]}= 2$ then there exists two non-zero product of $3$ elements. Call them $v$ and $w$. Then $v$ and $w$ are both of the form $c \cdot d$ where $c \in \{x,y\}$ and $d \in S\setminus A^{[3]}$ or vice-versa. Then $\{x,y,u,v,w\}$ forms a basis of the $\mathbb{F}_p$ vector space $A$.\\
Note for $i,j\in \{x,y\}$,
\[i \cdot j=\alpha_{ij}u+\beta_{ij}v+\gamma_{ij}w\] where $\alpha_{ij},\beta_{ij},\gamma_{ij} \in \mathbb{F}_p $. As $S \not\subset A^{[3]}$ not all of $\alpha_{ij}=0$.\\
Furthermore for $k,l$ with $k=u,l \in \{x,y\}$ or $k \in \{x,y\},l=u$, we let \[k\cdot l=\delta_{kl}v+\nu_{kl}w.\]
Note that the only two relevant pre-Lie algebra relations here are
\[(x \cdot y)\cdot x-x \cdot (y \cdot x)=(y \cdot x)\cdot x-y \cdot x^2\]
and 
\[(y \cdot x)\cdot y-y \cdot (x \cdot y)=(x \cdot y)\cdot y-x \cdot y^2\]
as products of more than $3$ elements are $0$. Using the notations specified we get the following relations among the scalars:
\begin{align*}
    \alpha_{xy}\delta_{ux}-\alpha_{yx}\delta_{xu}&=\alpha_{yx}\delta_{ux}-\alpha_{xx}\delta_{yu}.\\
    \alpha_{xy}\nu_{ux}-\alpha_{yx}\nu_{xu}&=\alpha_{yx}\nu_{ux}-\alpha_{xx}\nu_{yu}.\\
    \alpha_{yx}\delta_{uy}-\alpha_{xy}\delta_{yu}&=\alpha_{xy}\delta_{uy}-\alpha_{yy}\delta_{xu}.\\
    \alpha_{yx}\nu_{uy}-\alpha_{xy}\nu_{yu}&=\alpha_{xy}\nu_{uy}-\alpha_{yy}\nu_{xu}.
\end{align*}
We can summarize these by stating the following theorem.
\begin{theo}
     Let $A$ be a pre-Lie algebra generated by two elements $x$ and $y$ as a pre-Lie algebra. Assume $A^{[3]} \neq 0$, $A^{[4]}=0$ and $\dime_{\mathbb{F}_p}A^{[3]}=2$. Then we have the 
following
\begin{itemize}
    \item[1.] Not all of $S:=\{x \cdot x,x\cdot y, y \cdot x, y \cdot y\}$ is in $A^{[3]}$.
    \item[2.] $A$ is a $\mathbb{F}_p$ vector space with basis $\{x,y,u,v,w\}$ where $u \in S \setminus A^{[3]}$ and $v,w$ has one factor in $\{x,y\}$ and the other factor is $u$.
    \item[3.] $A^{[3]}=\mathbb{F}_pv+\mathbb{F}_pw$ and $A/A^{[2]}=\mathbb{F}_pu$.
    \item[4.] for $i,j\in \{x,y\}$,
\[i \cdot j=\alpha_{ij}u+\beta_{ij}v+\gamma_{ij}w\] where $\alpha_{ij},\beta_{ij},\gamma_{ij} \in \mathbb{F}_p $. As $S \not\subset A^{[3]}$ not all of $\alpha_{ij}=0$.\\
Furthermore for $k,l$ with $k=u,l \in \{x,y\}$ or $k \in \{x,y\},l=u$, we let \[k\cdot l=\delta_{kl}v+\nu_{kl}w.\]. Then we have 
\begin{align*}
    \alpha_{xy}\delta_{ux}-\alpha_{yx}\delta_{xu}&=\alpha_{yx}\delta_{ux}-\alpha_{xx}\delta_{yu}.\\
    \alpha_{xy}\nu_{ux}-\alpha_{yx}\nu_{xu}&=\alpha_{yx}\nu_{ux}-\alpha_{xx}\nu_{yu}.\\
    \alpha_{yx}\delta_{uy}-\alpha_{xy}\delta_{yu}&=\alpha_{xy}\delta_{uy}-\alpha_{yy}\delta_{xu}.\\
    \alpha_{yx}\nu_{uy}-\alpha_{xy}\nu_{yu}&=\alpha_{xy}\nu_{uy}-\alpha_{yy}\nu_{xu}.
\end{align*}
\item[5.] If $A$ satisfies $1-4$ it is a well defined pre-Lie algebra.
\end{itemize}
\end{theo}
\textbf{Case -II:}
Let $\dime_{\mathbb{F}_p}A^{[3]}=1$. By similar argument as in Case-I, $\dime_{\mathbb{F}_p}A/A^{[2]}=2$. Then $\dime_{\mathbb{F}_p}A^{[2]}/A^{[3]}=2$. Note $A^{[2]}$ is generated by  $S:=\{x \cdot x,x\cdot y, y \cdot x, y \cdot y\}$. Again all of $S$ can not be in $A^{[3]}$. So at least $2$ elements from $S$ should not be in $A^{[3]}$. Let $u$ and $v$ be $2$ such linearly independent elements in $S \setminus A^{[3]}$.\\
As $A^{[3]}\neq 0$, we ca pick a non-zero element of the form $c \cdot d$ where $c \in \{x,y\}$ and $d \in S \setminus A^{[3]}$ or vice-versa. Denote that element by $w$.\\
Then $\{x,y,u,v,w\}$ forms a basis of the ${\mathbb{F}_p}$ vector space $A$. Also for $i,j \in \{x,y\}$ we have 
\[i \cdot j=\alpha_{ij}u+\beta_{ij}v+\gamma_{ij}w\]Note all pairs $(\alpha_{ij},\beta_{ij})$ can not be $(0,0)$. Also for $k\in \{x,y\}$ and $l \in \{u,v\}$ or vice-versa we have
\[k \cdot l=\delta_{kl}w.\]
Again as in Case-I the two relevant pre-Lie algebra relations here are
\[(x \cdot y)\cdot x-x \cdot (y \cdot x)=(y \cdot x)\cdot x-y \cdot x^2\]
and 
\[(y \cdot x)\cdot y-y \cdot (x \cdot y)=(x \cdot y)\cdot y-x \cdot y^2\]
as products of more than $3$ elements are $0$. Using the notations specified we get the following relations among the scalars:
\begin{align*}
    \alpha_{xy}\delta_{ux}+\beta_{xy}\delta_{vx}-\alpha_{yx}\delta_{xu}-\beta_{yx}\delta_{xv}&=\alpha_{yx}\delta_{ux}+\beta_{yx}\delta_{vx}-\alpha_{xx}\delta_{yu}-\beta_{xx}\delta_{yv}.\\
    \alpha_{yx}\delta_{uy}+\beta_{yx}\delta_{vy}-\alpha_{xy}\delta_{yu}-\beta_{xy}\delta_{yv}&=\alpha_{xy}\delta_{uy}+\beta_{xy}\delta_{vy}-\alpha_{yy}\delta_{xu}-\beta_{yy}\delta_{xv}.
\end{align*}
Hence we have the following theorem
\begin{theo}
    Let $A$ be a pre-Lie algebra generated by two elements $x$ and $y$ as a pre-Lie algebra. Assume $A^{[3]} \neq 0$, $A^{[4]}=0$ and $\dime_{\mathbb{F}_p}A^{[3]}=1$. Then we have the 
following
\begin{itemize}
    \item[1.] Not all of $S:=\{x \cdot x,x\cdot y, y \cdot x, y \cdot y\}$ is in $A^{[3]}$.
    \item[2.] $A$ has a $\mathbb{F}_p$ basis $\{x,y,u,v,w\}$ where $u,v \in S\setminus A^{[3]}$ and $w$ is of the form $c \cdot d$ where $c\in \{x,y\}$ and $d \in \{u,v\}$ or vice-versa.
    \item[3.] $A^{[2]}/A^{[3]}=\mathbb{F}_pu+\mathbb{F}_pv$ and $A^{[3]}=\mathbb{F}_pw$.
    \item[4.] For $i,j \in \{x,y\}$ we have 
\[i \cdot j=\alpha_{ij}u+\beta_{ij}v+\gamma_{ij}w\]Note all pairs $(\alpha_{ij},\beta_{ij})$ can not be $(0,0)$. Also for $k\in \{x,y\}$ and $l \in \{u,v\}$ or vice-versa we have
\[k \cdot l=\delta_{kl}w.\] Then we have
\begin{align*}
    \alpha_{xy}\delta_{ux}+\beta_{xy}\delta_{vx}-\alpha_{yx}\delta_{xu}-\beta_{yx}\delta_{xv}&=\alpha_{yx}\delta_{ux}+\beta_{yx}\delta_{vx}-\alpha_{xx}\delta_{yu}-\beta_{xx}\delta_{yv}.\\
    \alpha_{yx}\delta_{uy}+\beta_{yx}\delta_{vy}-\alpha_{xy}\delta_{yu}-\beta_{xy}\delta_{yv}&=\alpha_{xy}\delta_{uy}+\beta_{xy}\delta_{vy}-\alpha_{yy}\delta_{xu}-\beta_{yy}\delta_{xv}.
\end{align*}
\item[5.] If $A$ satisfies $1-4$ it is a well defined pre-Lie algebra.
\end{itemize}
\end{theo}
\subsection{$A^{[4]} \neq 0$ and $A^{[5]}=0$.} From Theorem \ref{nil2} it is clear that $\dime_{\mathbb{F}_p}A^{[4]}=1$. Define $S:=\{x \cdot x,x\cdot y, y \cdot x, y \cdot y\}$. Note that $A^{[3]}$ is generated by $S':=\{x\cdot s,y \cdot s,s \cdot x,s \cdot y\}$. As $A^{[4]}\neq A^{[3]}$, $S' \not\subset A^{[4]}$. So there exists at least one non-zero element of the form $c \cdot d$ where $c \in \{x,y\}$, $d\in S$ or vice-versa such that $c \cdot d \notin A^{[4]}$. Denote that element by $v$ and the element of $S$ which is a factor $v$, denote it by $u$. Note $u \in S \setminus A^{[3]}$.\\
As $A^{[4]}$ is one dimensional, pick a non zero element say $w$ from $A^{[4]}$. Then $A^{[4]}=\mathbb{F}_pw, A^{[3]}/A^{[4]}=\mathbb{F}_pv$ and $A^{[2]}/A^{[3]}==\mathbb{F}_pu$.\\
So $\{x,y,u,v,w\}$ forms a basis of $A$. For $i,j \in \{x,y\}$ we have 
\[i \cdot j=\alpha_{ij}u+\beta_{ij}v+\gamma_{ij}w\]Note not all $\alpha_{ij}$ can be $0$. Also for $k\in \{x,y\}$ and $l =u$ or vice-versa we have
\[k \cdot l=\delta_{kl}v+\nu_{kl}w.\] Note not all $\delta_{kl}$ are zero.
Also let \[u\cdot u=\mu_{uu}w\] and for $m=v,n \in\{x,y\}$ or vice-versa let \[m\cdot n=\mu_{mn}w.\]
Again from $2$ relevant relations from $A^{[3]}$, we have
\begin{align*}
    \alpha_{xy}\delta_{ux}-\alpha_{yx}\delta_{xu}&=\alpha_{yx}\delta_{xu}-\alpha_{xx}\delta_{uy}.\\
    \alpha_{yx}\delta_{uy}-\alpha_{xy}\delta_{yu}&=\alpha_{xy}\delta_{yu}-\alpha_{yy}\delta_{ux}.\\
    \alpha_{xy}\nu_{ux}+\beta_{xy}\mu_{vx}-\alpha_{yx}\nu_{xu}-\beta_{yx}\mu_{xv}&=\alpha_{yx}\nu_{ux}+\beta_{yx}\mu_{vx}-\alpha_{xx}\nu_{uy}-\beta_{xx}\mu_{yv}.\\
       \alpha_{yx}\nu_{uy}+\beta_{yx}\mu_{vy}-\alpha_{xy}\nu_{yu}-\beta_{xy}\mu_{yv}&=\alpha_{xy}\nu_{uy}+\beta_{xy}\mu_{vy}-\alpha_{yy}\nu_{ux}-\beta_{yy}\mu_{xv}.
    `\end{align*}
There are $20$ relevant relations in $A^{[4]}$. Since we know there exists at least one $\alpha_{ij}\neq 0$, from those $20$ relations we have
\begin{align*}
\alpha_{xy}^2\mu_{uu}-\alpha_{xy}\delta_{yu}\mu_{xv}&=\alpha_{yx}\alpha_{xy}\mu_{uu}-\alpha_{xy}\delta_{xu}\mu_{yv}.\\
    \alpha_{xy}\alpha_{yx}\mu_{uu}-\alpha_{yx}\delta_{yu}\mu_{xv}&=\alpha_{yx}^2\mu_{uu}-\alpha_{yx}\delta_{xu}\mu_{yv}.\\
    \alpha_{xy}\alpha_{xx}\mu_{uu}-\alpha_{xx}\delta_{yu}\mu_{xv}&=\alpha_{yx}\alpha_{xx}\mu_{uu}-\alpha_{xx}\delta_{xu}\mu_{yv}.\\
    \alpha_{xy}\alpha_{yy}\mu_{uu}-\alpha_{yy}\delta_{yu}\mu_{xv}&=\alpha_{yx}\alpha_{yy}\mu_{uu}-\alpha_{yy}\delta_{xu}\mu_{yv}.\\
\delta_{xu}\mu_{vy}-\delta_{uy}\mu_{xv}&=\delta_{ux}\mu_{vy}-\alpha_{xy}\mu_{uu}.\\
\delta_{xu}\mu_{vx}- \delta_{ux}\mu_{xv}&=\delta_{ux}\mu_{vx}-\alpha_{xx}\mu_{uu}.\\
\delta_{yu}\mu_{vy}-\delta_{uy}\mu_{yv}&=\delta_{uy}\mu_{vy}-\alpha_{yy}\mu_{uu}.\\
    \delta_{yu}\mu_{vx}- \delta_{ux}\mu_{yv}&=\delta_{uy}\mu_{vx}-\alpha_{yx}\mu_{uu}.
\end{align*}
So we can summarize the section by the following theorem
\begin{theo}
    Let $A$ be a pre-Lie algebra generated by two elements $x$ and $y$ as a pre-Lie algebra. Assume $A^{[4]} \neq 0$ and $A^{[5]}=0$. Then we have the 
following
\begin{itemize}
    \item[1.] $\{x,y,u,v,w\}$ forms a $\mathbb{F}_p$ basis of $A$ where $u \in S\setminus A^{[3]}$, $S:=\{x \cdot x,x\cdot y, y \cdot x, y \cdot y\}$, $v$ is of the form $c \cdot d$ where $c\in \{x,y\}$ and $d=u$ or vice-versa and $w$ is any non-zero element in $A^{[4]}$.
  \item[2.] $A^{[2]}/A^{[3]}=\mathbb{F}_pu$, $A^{[3]}/A^{[4]}=\mathbb{F}_pv$ and $A^{[4]]}=\mathbb{F}_pw$.
    \item[4.] For $i,j \in \{x,y\}$ we have 
\[i \cdot j=\alpha_{ij}u+\beta_{ij}v+\gamma_{ij}w\]Not all $\alpha_{ij}$ can be $0$. Also for $k\in \{x,y\}$ and $l=u$ or vice-versa we have
\[k \cdot l=\delta_{kl}v+\nu_{kl}w.\] Also not all $\delta_{kl}$ are zero. Also \[u \cdot u=\mu_{uu}w\] and for $m=v,n \in\{x,y\}$ or vice-versa \[m\cdot n=\mu_{mn}w\] Then we have
\begin{align*}
     \alpha_{xy}\delta_{ux}-\alpha_{yx}\delta_{xu}&=\alpha_{yx}\delta_{xu}-\alpha_{xx}\delta_{uy}.\\
    \alpha_{yx}\delta_{uy}-\alpha_{xy}\delta_{yu}&=\alpha_{xy}\delta_{yu}-\alpha_{yy}\delta_{ux}.\\
    \alpha_{xy}\nu_{ux}+\beta_{xy}\mu_{vx}-\alpha_{yx}\nu_{xu}-\beta_{yx}\mu_{xv}&=\alpha_{yx}\nu_{ux}+\beta_{yx}\mu_{vx}-\alpha_{xx}\nu_{uy}-\beta_{xx}\mu_{yv}.\\
       \alpha_{yx}\nu_{uy}+\beta_{yx}\mu_{vy}-\alpha_{xy}\nu_{yu}-\beta_{xy}\mu_{yv}&=\alpha_{xy}\nu_{uy}+\beta_{xy}\mu_{vy}-\alpha_{yy}\nu_{ux}-\beta_{yy}\mu_{xv}.\\
       \alpha_{xy}^2\mu_{uu}-\alpha_{xy}\delta_{yu}\mu_{xv}&=\alpha_{yx}\alpha_{xy}\mu_{uu}-\alpha_{xy}\delta_{xu}\mu_{yv}.\\
    \alpha_{xy}\alpha_{yx}\mu_{uu}-\alpha_{yx}\delta_{yu}\mu_{xv}&=\alpha_{yx}^2\mu_{uu}-\alpha_{yx}\delta_{xu}\mu_{yv}.\\
    \alpha_{xy}\alpha_{xx}\mu_{uu}-\alpha_{xx}\delta_{yu}\mu_{xv}&=\alpha_{yx}\alpha_{xx}\mu_{uu}-\alpha_{xx}\delta_{xu}\mu_{yv}.\\
    \alpha_{xy}\alpha_{yy}\mu_{uu}-\alpha_{yy}\delta_{yu}\mu_{xv}&=\alpha_{yx}\alpha_{yy}\mu_{uu}-\alpha_{yy}\delta_{xu}\mu_{yv}.\\
\delta_{xu}\mu_{vy}-\delta_{uy}\mu_{xv}&=\delta_{ux}\mu_{vy}-\alpha_{xy}\mu_{uu}.\\
\delta_{xu}\mu_{vx}- \delta_{ux}\mu_{xv}&=\delta_{ux}\mu_{vx}-\alpha_{xx}\mu_{uu}.\\
\delta_{yu}\mu_{vy}-\delta_{uy}\mu_{yv}&=\delta_{uy}\mu_{vy}-\alpha_{yy}\mu_{uu}.\\
    \delta_{yu}\mu_{vx}- \delta_{ux}\mu_{yv}&=\delta_{uy}\mu_{vx}-\alpha_{yx}\mu_{uu}.
\end{align*}
\item[5.] If $A$ satisfies $1-4$ it is a well defined pre-Lie algebra.
\end{itemize}
\end{theo}
\subsubsection{$A^{[4]}=A^{[5]}\neq 0$ and $A^{[6]}=0$.}
Note $A^{[4]}=A^{[5]}$ and $A^{[6]}=0$ implies $A \cdot A^{[4]}=A \cdot A^{5]}\subset A^{[6]}=0$. Due to similar reason $A^{[5]}\cdot A=0$. Also for $a \in A^{[3]},b,c \in A$, $a\cdot(b\cdot c) \in A^{[4]}\cdot A+A \cdot A^{[4]}=0$. So $A^{[3]}\cdot A^{[2]}=0$. If we take $a \in A^{[2]}, b \in A,c \in A^{[2]}$ we have $a\cdot(b\cdot c) \in A^{[3]}\cdot A^{[2]}+A \cdot A^{[4]}=0$. So the only possible non zero elements in $A^{[5]}$ is of the form $a\cdot(b\cdot c)$ where $a,b \in A^{[2]}, c \in A $.\\
From Theorem \ref{nil2}, it is clear that $\dime_{\mathbb{F}_p}A^{[4]}=\dime_{\mathbb{F}_p}A^{[5]}=1$. Since $A^{[5]}\neq 0$, we can choose any non-zero element of the form $c \cdot (d \cdot e)$ where $c,d \in S:=\{x \cdot x,x\cdot y, y \cdot x, y \cdot y\}$ and $e \in \{x,y\}$. Call that element $w$. So $A^{[5]}=A^{[4]}=\mathbb{F}_pw$.\\
Now $S \not\subset A^{[3]}$. Also every element of $A^{[3]}$ is of the form $c \cdot d$ where $c \in S,d \in \{x,y\}$ or vice-versa. As $\dime_{\mathbb{F}_p}A^{[3]}/A^{[4]}=1$, choose a non zero element of the form $c \cdot d$ where $c \in S\setminus A^{[3]}, d\in\{x,y\}$ or vice-versa. Call that elemnt $v$. Then $A^{[3]}/A^{[4]}=\mathbb{F}_pv$ and $A^{[3]}=\mathbb{F}_pv+\mathbb{F}_pw$.\\
Again $A^{[2]}/A^{[3]}$ is of dimension $1$. Choose a non-zero element from $S\setminus A^{[3]}$. Call it $u$. Then $\{x,y,u,v,w\}$ is a basis for $A$. For $i,j \in \{x,y\}$ we have 
\[i \cdot j=\alpha_{ij}u+\beta_{ij}v+\gamma_{ij}w\]Not all $\alpha_{ij}$ can be $0$. Also for $k\in \{x,y\}$ and $l=u$ or vice-versa we have
\[k \cdot l=\delta_{kl}v+\nu_{kl}w.\] Also not all $\delta_{kl}$ are zero. Also \[u \cdot u=\mu_{uu}w, u \cdot v=\mu_{uv}w\] and for $m=v,n \in\{x,y\}$ or vice-versa \[m\cdot n=\mu_{mn}w\] Note $v$ can not be of the form $c\cdot d$ where $c\in \{x,y\}, d \in S\setminus A^{[3]}$. If $v$ is of such form then $v \in A \cdot A^{[2]}$. Note all element of the of $A^{[5]}$ except elements of the form $A^{[2]}\cdot(A^{[2]}\cdot A)$ are $0$. Now elements of $A^{[3]}$ are of the form $\alpha~v+\beta~w$ where $\alpha,\beta \in \mathbb{F}_p$. So, for any $a\in A^{[2]}$,
\[a\cdot(\alpha~v+\beta~w)=\alpha~a\cdot v+\beta~a \cdot w=0\] as $a\cdot w \in A^{[7]}=0$ and $a\cdot v$ is of the form $A^{[2]}\cdot(A\cdot A^{[2]})=0$. So $A^{[2]} \cdot A^{[3]}=0$ implying $A^{[5]}=0$, a contradiction.\\
From the relations in $A^{[3]}$ we have the same set of relations between the scalars as obtained in Subsection $5.3$.
Now from the relations in $A^{[4]}$ along with the fact that at least one $\alpha_{ij}$ is non zero we have
\begin{align*}
    \alpha_{xy}^2\mu_{uu}+\alpha_{xy}\beta_{xy}\mu_{uv}-\alpha_{xy}\delta_{yu}\mu_{xv}&=\alpha_{yx}\alpha_{xy}\mu_{uu}+\alpha_{yx}\beta_{xy}\mu_{uv}-\alpha_{xy}\delta_{xu}\mu_{yv}.\\
\alpha_{xy}\alpha_{yx}\mu_{uu}+\alpha_{xy}\beta_{yx}\mu_{uv}-\alpha_{yx}\delta_{yu}\mu_{xv}&=\alpha_{yx}^2\mu_{uu}+\alpha_{yx}\beta_{yx}\mu_{uv}-\alpha_{yx}\delta_{xu}\mu_{yv}.\\
 \alpha_{xy}\alpha_{xx}\mu_{uu}+\alpha_{xy}\beta_{xx}\mu_{uv}-\alpha_{xx}\delta_{yu}\mu_{xv}&=\alpha_{yx}\alpha_{xx}\mu_{uu}+\alpha_{yx}\beta_{xx}\mu_{uv}-\alpha_{xx}\delta_{xu}\mu_{yv}.\\
  \alpha_{xy}\alpha_{yy}\mu_{uu}+\alpha_{xy}\beta_{yy}\mu_{uv}-\alpha_{yy}\delta_{yu}\mu_{xv}&=\alpha_{yx}\alpha_{yy}\mu_{uu}+\alpha_{yx}\beta_{yy}\mu_{uv}-\alpha_{yy}\delta_{xu}\mu_{yv}.\\
 \delta_{xu}\mu_{vy}-\delta_{uy}\mu_{xv}&=\delta_{ux}\mu_{vy}-\alpha_{xy}\mu_{uu}-\beta_{xy}\mu_{uv}.\\
\delta_{xu}\mu_{vx}- \delta_{ux}\mu_{xv}&=\delta_{ux}\mu_{vx}-\alpha_{xx}\mu_{uu}-\beta_{xx}\mu_{uv}.\\
 \delta_{yu}\mu_{vy}- \delta_{uy}\mu_{yv}&=\delta_{uy}\mu_{vy}-\alpha_{yy}\mu_{uu}-\beta_{yy}\mu_{uv}.\\
\delta_{yu}\mu_{vx}- \delta_{ux}\mu_{yv}&=\delta_{uy}\mu_{vx}-\alpha_{yx}\mu_{uu}-\beta_{yx}\mu_{uv}.    
\end{align*}
So now we can summarize the results in the following theorem.
\begin{theo}
   Let $A$ be a pre-Lie algebra generated by two elements $x$ and $y$ as a pre-Lie algebra. Assume $A^{[4]}=A^{[5]}\neq 0$ and $A^{[6]}=0$. Then we have the 
following
\begin{itemize}
    \item[1.] $\{x,y,u,v,w\}$ forms a $\mathbb{F}_p$ basis of $A$ where $u \in S\setminus A^{[3]}$, $S:=\{x \cdot x,x\cdot y, y \cdot x, y \cdot y\}$, $v$ is of the form $c \cdot d$ where $c\in \{x,y\}$ and $d \in S\setminus A^{[3]}$ or vice-versa and $w$ is any non-zero element in $A^{[5]}$ of the form $e \cdot(f \cdot g)$ where $e,f \in S\setminus A^{[3]}$ and $g \in\{x,y\}$.
  \item[2.] $A^{[2]}/A^{[3]}=\mathbb{F}_pu$, $A^{[3]}/A^{[4]}=\mathbb{F}_pv$ and $A^{[4]}=A^{[5]}=\mathbb{F}_pw$.
    \item[3.] For $i,j \in \{x,y\}$ we have 
\[i \cdot j=\alpha_{ij}u+\beta_{ij}v+\gamma_{ij}w\]Not all $\alpha_{ij}$ can be $0$. Also for $k\in \{x,y\}$ and $l=u$ or vice-versa we have
\[k \cdot l=\delta_{kl}v+\nu_{kl}w.\] Also not all $\delta_{kl}$ are zero. Also \[u \cdot u=\mu_{uu}w,u \cdot v=\mu_{uv}w\] and for $m=v,n \in\{x,y\}$ or vice-versa \[m\cdot n=\mu_{mn}w\] Then we have
\begin{align*}
     \alpha_{xy}\delta_{ux}-\alpha_{yx}\delta_{xu}&=\alpha_{yx}\delta_{xu}-\alpha_{xx}\delta_{uy}.\\
    \alpha_{yx}\delta_{uy}-\alpha_{xy}\delta_{yu}&=\alpha_{xy}\delta_{yu}-\alpha_{yy}\delta_{ux}.\\
    \alpha_{xy}\nu_{ux}+\beta_{xy}\mu_{vx}-\alpha_{yx}\nu_{xu}-\beta_{yx}\mu_{xv}&=\alpha_{yx}\nu_{ux}+\beta_{yx}\mu_{vx}-\alpha_{xx}\nu_{uy}-\beta_{xx}\mu_{yv}.\\
       \alpha_{yx}\nu_{uy}+\beta_{yx}\mu_{vy}-\alpha_{xy}\nu_{yu}-\beta_{xy}\mu_{yv}&=\alpha_{xy}\nu_{uy}+\beta_{xy}\mu_{vy}-\alpha_{yy}\nu_{ux}-\beta_{yy}\mu_{xv}.\\
       \alpha_{xy}^2\mu_{uu}+\alpha_{xy}\beta_{xy}\mu_{uv}-\alpha_{xy}\delta_{yu}\mu_{xv}&=\alpha_{yx}\alpha_{xy}\mu_{uu}+\alpha_{yx}\beta_{xy}\mu_{uv}-\alpha_{xy}\delta_{xu}\mu_{yv}.\\
\alpha_{xy}\alpha_{yx}\mu_{uu}+\alpha_{xy}\beta_{yx}\mu_{uv}-\alpha_{yx}\delta_{yu}\mu_{xv}&=\alpha_{yx}^2\mu_{uu}+\alpha_{yx}\beta_{yx}\mu_{uv}-\alpha_{yx}\delta_{xu}\mu_{yv}.\\
 \alpha_{xy}\alpha_{xx}\mu_{uu}+\alpha_{xy}\beta_{xx}\mu_{uv}-\alpha_{xx}\delta_{yu}\mu_{xv}&=\alpha_{yx}\alpha_{xx}\mu_{uu}+\alpha_{yx}\beta_{xx}\mu_{uv}-\alpha_{xx}\delta_{xu}\mu_{yv}.\\
  \alpha_{xy}\alpha_{yy}\mu_{uu}+\alpha_{xy}\beta_{yy}\mu_{uv}-\alpha_{yy}\delta_{yu}\mu_{xv}&=\alpha_{yx}\alpha_{yy}\mu_{uu}+\alpha_{yx}\beta_{yy}\mu_{uv}-\alpha_{yy}\delta_{xu}\mu_{yv}.\\
\delta_{xu}\mu_{vy}-\delta_{uy}\mu_{xv}&=\delta_{ux}\mu_{vy}-\alpha_{xy}\mu_{uu}-\beta_{xy}\mu_{uv}.\\
\delta_{xu}\mu_{vx}- \delta_{ux}\mu_{xv}&=\delta_{ux}\mu_{vx}-\alpha_{xx}\mu_{uu}-\beta_{xx}\mu_{uv}.\\
 \delta_{yu}\mu_{vy}- \delta_{uy}\mu_{yv}&=\delta_{uy}\mu_{vy}-\alpha_{yy}\mu_{uu}-\beta_{yy}\mu_{uv}.\\
\delta_{yu}\mu_{vx}- \delta_{ux}\mu_{yv}&=\delta_{uy}\mu_{vx}-\alpha_{yx}\mu_{uu}-\beta_{yx}\mu_{uv}.  
\end{align*}
\item[5.] If $A$ satisfies $1-3$ it is a well defined pre-Lie algebra.
\end{itemize}  
\end{theo}
\section{pre-Lie algebra $A$ with $(A,+)\cong C_p^5$ and generated by three elements as a pre-Lie ring}\label{s2}
Let us start by commenting on the nilpotency index of $(A,+,\cdot)$ which is generated by $3$ elements as a pre-Lie algebra.
\begin{theo}\label{nil3}
    Let $A$ be a pre-Lie algebra generated by $3$ elements. Then $A^{[2]}\neq 0$ and $A^{[4]}=0$.
\end{theo}
\begin{proof}
    Clearly $A^{[2]}\neq 0$. Otherwise $\dime_{\mathbb{F}_p}\leq 3$, a contradiction.\\
    Now as $A$ has $3$ generators as a pre-Lie algebra, $A/A^{[2]}$ has dimension $3$ as a $\mathbb{F}_p$ vector space. So $\dime_{\mathbb{F}_p}A^{[2]}=2$. Then as $A^{[3]} \neq A^{[2]}$, we have $\dime_{\mathbb{F}_p}A^{[2]} \leq 1$. If $\dime_{\mathbb{F}_p}A^{[3]}=0$, then $A^{[3]}=0$ and we are done. If $\dime_{\mathbb{F}_p}A^{[3]}=1$, then $A^{[4]}\neq A^{[3]}$ implies $A^{[4]}=0$. 
\end{proof}
\subsection{$A^{[2]}\neq 0$ and $A^{[3]}=0$.} Let us assume $A$ is generated by $x,y,z$ as a pre-Lie algebra. Then $A^{[2]}$ is a $\mathbb{F}_p$ vector space generated by $S':=\{x \cot x,x\cdot y,x\cdot z,y \cdot x,y\cdot y,y\cdot z,z\cdot x,z\cdot y, z \cdot z\}$. Choose any two no-zero linearly independent vectors from $S'$, say $u$ and $v$. Then $\{x,y,z,u,v\}$ forms a basis of the $\mathbb{F}_p$ vector space $A$. Now for $i,j\in \{x,y,z\}$, let \[i\cdot j=\alpha_{ij}u+\beta_{ij}v.\] Note that at least two tuples $(\alpha_{ij},\beta_{ij})$ are not $(0,0)$ and multiplies of each other. So we can summarize the result in the following theorem.
\begin{theo}
    Let $(A,+,\cdot)$ be a pre-Lie algebra generated by $3$ elements $x,y,z$ as a pre-Lie algebra. Suppose $A^{[2]}\neq 0$ and $A^{[3]}=0$. Then we have the following
    \begin{itemize}
        \item[1.] $\{x,y,z,u,v\}$ forms a basis of the $\mathbb{F}_p$ vector space $A$ where $u,v \in A^{[2]}$.
        \item[2.] $A^{[2]}=\mathbb{F}_pu+\mathbb{F}_pv$.
        \item[3.]  For $i,j\in \{x,y,z\}$, let \[i\cdot j=\alpha_{ij}u+\beta_{ij}v.\] where at least two tuples $(\alpha_{ij},\beta_{ij})$ are not $(0,0)$ and multiplies of each other.
    \end{itemize}
\end{theo}
\subsection{$A^{[3]} \neq 0$ and $A^{[4]}=0$.}We already have $\dime_{\mathbb{F}_p}A/A^{[2]}=3$ implying $\dime_{\mathbb{F}_p}A^{[2]}=2$. This implies from Theorem \ref{nil3} $\dime_{\mathbb{F}_p}A^{[3]}=1$.\\
Note all of $S':=\{x \cdot x,x\cdot y,x\cdot z,y \cdot x,y\cdot y,y\cdot z,z\cdot x,z\cdot y, z \cdot z\}$ can not be contained in $A^{[3]}$.
Take a non zero element $u \in S'\setminus A^{[3]}$.\\
Now as $\dime_{\mathbb{F}_p}A^{[3]}=1$, take a non zero element say $v$ from $A^{[3]}$ which will be of the form $c \cdot d$ where $c \in S'\setminus A^{[3]}$ and $d\in \{x,y\}$ or vice-versa. Then $\{x,y,z,u,v\}$ forms a basis of the $\mathbb{F}_p$ vector space $A$. Clearly $A^{[3]}=\mathbb{F}_pv$ and $A^{[2]}/A^{[3]}=\mathbb{F}_pu$.\\
Again for $i,j\in \{x,y,z\}$, let \[i\cdot j=\alpha_{ij}u+\beta_{ij}v.\] Here not all $\alpha_{ij}$ are $0$. Also for $k=u, l\in\{x,y,z\}$ or vice-versa, define \[k \cdot l=\delta_{kl}v.\]Now the relevant pre-Lie relations in $A^{[3]}$ are 
\begin{align*}
    (x \cdot y)\cdot x-x \cdot (y \cdot x)=(y \cdot x)\cdot x-y \cdot x^2, ~&
(y \cdot x)\cdot y-y \cdot (x \cdot y)=(x \cdot y)\cdot y-x \cdot y^2\\
(x \cdot z)\cdot x-x \cdot (z \cdot x)=(z \cdot x)\cdot x-z \cdot x^2,~&
(z \cdot x)\cdot z-z \cdot (x \cdot z)=(x \cdot z)\cdot z-x \cdot z^2.\\
(y \cdot z)\cdot y-y \cdot (z \cdot y)=(z \cdot y)\cdot y-z \cdot y^2,~&
(z \cdot y)\cdot z-z \cdot (y \cdot z)=(y \cdot z)\cdot z-y \cdot z^2,\\
(x \cdot y)\cdot z-x \cdot (y \cdot z)=(y \cdot x)\cdot z-y \cdot (x\cdot z),~&
(x \cdot z)\cdot y-x \cdot (z \cdot y)=(z \cdot x)\cdot y-z \cdot (x\cdot y).\\
(y \cdot z)\cdot x-y \cdot (z \cdot x)&=(z \cdot y)\cdot x-z \cdot (y \cdot x).
\end{align*}
So we get the following relations among the scalars
\begin{align*}
    \alpha_{xy}\delta_{ux}-\alpha_{yx}\delta{xu}=\alpha_{yx}\delta_{ux}-\alpha_{xx}\delta_{yu},~&
    \alpha_{yx}\delta_{uy}-\alpha_{xy}\delta{yu}=\alpha_{xy}\delta_{uy}-\alpha_{yy}\delta_{xu}.\\
    \alpha_{xz}\delta_{ux}-\alpha_{zx}\delta{xu}=\alpha_{zx}\delta_{ux}-\alpha_{xx}\delta_{zu},~&
    \alpha_{zx}\delta_{uz}-\alpha_{xz}\delta{zu}=\alpha_{xz}\delta_{uz}-\alpha_{zz}\delta_{xu}.\\
    \alpha_{yz}\delta_{uy}-\alpha_{zy}\delta{yu}=\alpha_{zy}\delta_{uy}-\alpha_{yy}\delta_{zu},~&
    \alpha_{zy}\delta_{uz}-\alpha_{yz}\delta{zu}=\alpha_{yz}\delta_{uz}-\alpha_{zz}\delta_{yu}.\\
    \alpha_{xy}\delta_{uz}-\alpha_{yz}\delta{xu}=\alpha_{yx}\delta_{uz}-\alpha_{xz}\delta_{yu},~&
    \alpha_{yz}\delta_{ux}-\alpha_{zx}\delta{yu}=\alpha_{zy}\delta_{ux}-\alpha_{yx}\delta_{zu}.\\
    \alpha_{xz}\delta_{uy}-\alpha_{zy}\delta{xu}&=\alpha_{zx}\delta_{uy}-\alpha_{xy}\delta_{zu}.
\end{align*}
So we summarize these results in the following theorem.
\begin{theo}
    Let $A$ be a pre-Lie algebra generated by $x,y,z$. Suppose $A^{[3]}\neq 0$ and $A^{[4]}=0$. Then we have the following
    \begin{itemize}
        \item[1.] $\{x,y,z,u,v\}$ forms a basis of the $\mathbb{F}_p$ vector space $A$ where $u \in S'\setminus A^{[3]}$ and $v$ is of the form $c \cdot d$ where $c\in S'\setminus A^{[3]}, d \in \{x,y,z\}$ or vice-versa where $S':=\{x \cdot x,x\cdot y,x\cdot z,y \cdot x,y\cdot y,y\cdot z,z\cdot x,z\cdot y, z \cdot z\}$. 
        \item[2.] $A^{[2]}/A^{[3]}=\mathbb{F}_pu, A^{[3]}=\mathbb{F}_pv$.
        \item[3.]  for $i,j\in \{x,y,z\}$, let \[i\cdot j=\alpha_{ij}u+\beta_{ij}v.\] Here not all $\alpha_{ij}$ are $0$. Also for $k=u, l\in\{x,y,z\}$ or vice-versa, define \[k \cdot l=\delta_{kl}v.\] Then we have the following
        \begin{align*}
             \alpha_{xy}\delta_{ux}-\alpha_{yx}\delta{xu}=\alpha_{yx}\delta_{ux}-\alpha_{xx}\delta_{yu},~&
    \alpha_{yx}\delta_{uy}-\alpha_{xy}\delta{yu}=\alpha_{xy}\delta_{uy}-\alpha_{yy}\delta_{xu}.\\
    \alpha_{xz}\delta_{ux}-\alpha_{zx}\delta{xu}=\alpha_{zx}\delta_{ux}-\alpha_{xx}\delta_{zu},~&
    \alpha_{zx}\delta_{uz}-\alpha_{xz}\delta{zu}=\alpha_{xz}\delta_{uz}-\alpha_{zz}\delta_{xu}.\\
    \alpha_{yz}\delta_{uy}-\alpha_{zy}\delta{yu}=\alpha_{zy}\delta_{uy}-\alpha_{yy}\delta_{zu},~&
    \alpha_{zy}\delta_{uz}-\alpha_{yz}\delta{zu}=\alpha_{yz}\delta_{uz}-\alpha_{zz}\delta_{yu}.\\
    \alpha_{xy}\delta_{uz}-\alpha_{yz}\delta{xu}=\alpha_{yx}\delta_{uz}-\alpha_{xz}\delta_{yu},~&
    \alpha_{yz}\delta_{ux}-\alpha_{zx}\delta{yu}=\alpha_{zy}\delta_{ux}-\alpha_{yx}\delta_{zu}.\\
    \alpha_{xz}\delta_{uy}-\alpha_{zy}\delta{xu}&=\alpha_{zx}\delta_{uy}-\alpha_{xy}\delta_{zu}.
        \end{align*}
        \item[4.] If $A$ satisfies $1-3$ then it is a well defined pre-Lie algbera. 
    \end{itemize}
\end{theo}
\section{pre-Lie algebra $A$ with $(A,+)\cong C_p^5$ and generated by four elements as a pre-Lie ring}
We start this section by the following theorem.
\begin{theo}
    Let $A$ be a pre-Lie algebra generated by $4$ elements. Then $A^{[2]} \neq 0$ and $A^{[3]}=0$.
\end{theo}
\begin{proof}
    Note if $A^{[2]}=0$ then $\dime_{\mathbb{F}_p}A \leq 4$. So $A^{[2]} \neq 0$. Also $\dime_{\mathbb{F}_p} A/A^{[2]}=4$ implying $\dime_{\mathbb{F}_p}A^{[2]}=1$. As $A^{[3]}\neq A^{[2]}$ we have $A^{[3]}=0$.
\end{proof}
Here we only have a single case: $A^{[2]} \neq 0$ and $A^{[3]}=0$.\\
Let $A$ be generated by $x,y,z,u$ as a pre-Lie algebra. Then $S":=\{x\cdot x, x \cdot y,x\cdot z,x \cdot u,y\cdot x, y \cdot y,y\cdot z,y \cdot u,z\cdot x, z \cdot y,z\cdot z,z \cdot u,u\cdot x, u \cdot y,u\cdot z,u \cdot u\}$ generates $A^{[2]}$ as a $\mathbb{F}_p$ vector space. Take any non zero elemnt say $v$ from $S''$. Then $\{x,y,z,u,v\}$ forms a basis of $\mathbb{F}_p$ vector space $A$.\\
Also for $i,j\in\{x,y,z,u\}$ let \[i\cdot j=\alpha_{ij}v.\] Note all $\alpha_{ij}$ are not zero as $A^{[2]}\neq 0$. Thus we have
\begin{theo}
    Let $A$ be generated by $x,y,z,u$ as a pre-Lie algebra. Then
    \begin{itemize}
        \item[1.]  $A^{[2]} \neq 0$ and $A^{[3]}=0$.
        \item[2.]  $\{x,y,z,u,v\}$ forms a basis of $\mathbb{F}_p$ vector space $A$ where $u$ is a non zero element from $S'':=\{x\cdot x, x \cdot y,x\cdot z,x \cdot u,y\cdot x, y \cdot y,y\cdot z,y \cdot u,z\cdot x, z \cdot y,z\cdot z,z \cdot u,u\cdot x, u \cdot y,u\cdot z,u \cdot u\}$.
        \item[3.]  For $i,j\in\{x,y,z,u\}$ let \[i\cdot j=\alpha_{ij}v.\] where all $\alpha_{ij}$ are not zero.
    \end{itemize}
\end{theo}
\section{Concluding Remarks} In summary, this article makes significant contributions to the examination of braces of order $p^5$. We have successfully provided an explicit description of all strongly nilpotent pre-Lie algebras with a cardinality of $p^5$, consequently detailing all right nilpotent $\mathbb{F}_p$-braces with the same cardinality.
\par It is noteworthy that the same procedure can be employed to construct right nilpotent braces of order $p^5$ that are not necessarily $\mathbb{F}_p$-braces. In such cases, we consider $(A,+)$ to be isomorphic with various groups like $C_{p^4}\times C_p, C_{p^3}\times C_{p^2}, C_{p^3}\times C_p\times C_p, C_{p^2}\times C_{p^2} \times C_p, C_{p^2}\times C_p\times C_p\times C_p$. It's important to note that if $(A,+)$ is isomorphic to $C_{p^5}$, the corresponding brace becomes a cyclic brace, a topic well-explored in Rump's work \cite{r3}. Future research endeavors may focus on the classification of pre-Lie rings whose additive groups are isomorphic to the aforementioned groups.
\section*{Acknowledgement} The research of the author is financially supported by the National Board of Higher Mathematics under grant number 0203/34/2018/R\&D-II/139716. The author acknowledges Professor Manoj Kumar Yadav for initiating the problem, providing supervision throughout the process, and reviewing the solution. Additionally, the author expresses gratitude to the Harish Chandra Research Institute of Prayagraj for their warm hospitality during the months of August 2022 and September 2023.
\section*{Data Availability Statement}
Data sharing does not apply to this article as no datasets were generated or analyzed during the current study.

\end{document}